\documentclass[a4paper,11pt]{article}
\usepackage{amsmath, amsthm, amssymb,amsfonts}
\usepackage{bm}
\usepackage{graphicx}
\usepackage[colorlinks,linkcolor=blue,citecolor=blue,urlcolor=magenta,linktocpage,plainpages=false]{hyperref}
\usepackage{color}
\usepackage{fullpage}

\usepackage{url}
\usepackage{multicol}
\usepackage{caption}
\usepackage{subcaption}
\usepackage[normalem]{ulem}
\usepackage{thm-restate}
\usepackage{xspace}

\usepackage[title]{appendix}%
\usepackage{xcolor}%
\usepackage{booktabs}%
\usepackage{algorithm}%
\usepackage{algorithmicx}%
\usepackage{algpseudocode}%
\usepackage{comment}

\usepackage{enumerate}
\usepackage{enumitem}
\usepackage{mdframed}

\newtheorem{theorem}{Theorem}
\newtheorem{example}[theorem]{Example}%
\newtheorem{remark}[theorem]{Remark}%

\newtheorem{lemma}[theorem]{Lemma}
\newtheorem{observation}[theorem]{Observation}
\newtheorem{corollary}[theorem]{Corollary}
\newtheorem{claim}[theorem]{Claim}
\newtheorem{definition}[theorem]{Definition}

\newtheorem{invariant}{Invariant}

\theoremstyle{example}
\newtheorem{proc}{Procedure}
\theoremstyle{plain}

\newcommand{\bb}{\mathbb}
\newcommand{\R}{\bb R}
\newcommand{\Z}{{\bb Z}}
\newcommand{\N}{{\bb N}}

\newcommand*{\sub}{\textup{sub}}
\newcommand*{\sepp}{\textup{sep}}
\newcommand*{\update}{\textup{update}}
\newcommand*{\query}{\textup{query}}
\newcommand{\amitabh}[1]{\textcolor{red}{[Amitabh: #1]}}

\DeclareMathOperator*{\argmin}{argmin}

\DeclareMathOperator*{\dist}{dist}
\DeclareMathOperator*{\vol}{vol}

\DeclareMathOperator*{\icomp}{icomp}
\DeclareMathOperator*{\adv}{Adv}
\DeclareMathOperator*{\alg}{Alg}

\newcommand{\phillip}[1]{\textcolor{teal}{[Phillip: #1]}}

\def\ve#1{\mathchoice{\mbox{\boldmath$\displaystyle\bf#1$}}
{\mbox{\boldmath$\textstyle\bf#1$}}
{\mbox{\boldmath$\scriptstyle\bf#1$}}
{\mbox{\boldmath$\scriptscriptstyle\bf#1$}}}

\newcommand{\x}{{\ve x}}
\newcommand{\y}{{\ve y}}
\newcommand{\z}{{\ve z}}
\newcommand{\vv}{{\ve v}}
\newcommand{\g}{{\ve g}}
\newcommand{\e}{{\ve e}}
\renewcommand{\u}{{\ve u}}
\renewcommand{\a}{{\ve a}}

\newcommand{\0}{{\ve 0}}

\newcommand{\p}{{\ve p}}

\newcommand{\w}{{\ve w}}
\renewcommand{\b}{{\ve b}}

\newcommand{\cU}{\mathcal{U}}

\newcommand{\I}{\mathcal{I}}

\newcommand{\Lag}{\mathcal{L}}

\newcommand{\cH}{\mathcal{H}}
\newcommand{\cI}{\mathcal{I}}
\newcommand{\cO}{\mathcal{O}}

\newcommand{\F}{\mathcal{F}}
\newcommand{\cA}{\mathcal{A}}
\newcommand{\cM}{\mathcal{M}}
\newtheorem{prop}[theorem]{Proposition}

\newcommand{\red}[1]{{\color{red} #1}}
\newcommand{\OPT}{\textup{OPT}}
\newcommand{\mspan}{\textup{span}}
\newcommand{\ip}[2]{\langle #1, #2 \rangle}
\renewcommand{\epsilon}{\varepsilon}
\newcommand{\ones}{\mathbf{1}}
\newcommand{\cG}{\mathcal{G}}
\newcommand{\val}{\textup{val}}

\newcommand{\mnote}[1]{\red{M: #1}}
\newcommand{\teal}[1]{{\color{teal}#1}}

\usepackage{authblk}
\newcommand{\email}[1]{{\small{{\texttt{#1}}}}}
\newcommand{\tF}{\texttt{F}}
\newcommand{\zero}{\bm{0}}

\makeatletter
\newsavebox\myboxA
\newsavebox\myboxB
\newlength\mylenA
\newcommand*\xoverline[2][0.75]{%
    \sbox{\myboxA}{$\m@th#2$}%
    \setbox\myboxB\null
    \ht\myboxB=\ht\myboxA%
    \dp\myboxB=\dp\myboxA%
    \wd\myboxB=#1\wd\myboxA
    \sbox\myboxB{$\m@th\overline{\copy\myboxB}$}
    \setlength\mylenA{\the\wd\myboxA}
    \addtolength\mylenA{-\the\wd\myboxB}%
    \ifdim\wd\myboxB<\wd\myboxA%
       \rlap{\hskip 1\mylenA\usebox\myboxB}{\usebox\myboxA}%
    \else
        \hskip -1\mylenA\rlap{\usebox\myboxA}{\hskip 0.5\mylenA\usebox\myboxB}%
    \fi}
\makeatother

\let\eps\varepsilon

\newtheorem{conj}{Conjecture}
\begin{document}

\title{Information Complexity of Mixed-integer Convex Optimization}




\author[1]{Amitabh Basu}
\author[2]{Hongyi Jiang}
\author[1]{Phillip Kerger}
\author[3]{Marco Molinaro}

\affil[1]{Department of Applied Mathematics and Statistics, Johns Hopkins University\\
\email{\{abasu9,pkerger\}@jhu.edu} }
\affil[2]{School of Civil and Environmental Engineering, Cornell University\\
\email{hj348@cornell.edu} }
\affil[3]{Microsoft Research (Redmond) and Computer Science Department, PUC-Rio\\
\email{mmolinaro@microsoft.com}} 

\date{}

\maketitle

\begin{abstract}
We investigate the information complexity of mixed-integer convex optimization under different types of oracles. We establish new lower bounds for the standard first-order oracle, improving upon the previous best known lower bound. This leaves only a lower order linear term (in the dimension) as the gap between the lower and upper bounds. This is derived as a corollary of a more fundamental ``transfer" result that shows how lower bounds on information complexity of continuous convex optimization under different oracles can be transferred to the mixed-integer setting in a black-box manner. 

Further, we (to the best of our knowledge) initiate the study of, and obtain the first set of results on, information complexity under oracles that only reveal \emph{partial} first-order information, e.g., where one can only make a binary query over the function value or subgradient at a given point. We give algorithms for (mixed-integer) convex optimization that work under these less informative oracles. We also give lower bounds showing that, for some of these oracles, every algorithm requires more iterations to achieve a target error compared to when complete first-order information is available. That is, these oracles are provably less informative than full first-order oracles for the purpose of optimization.  

\vspace{0.3cm}
\noindent\textit{Keywords: Mixed-integer optimization, convex optimization, information complexity, lower bounds}
\end{abstract}

%
%
%
%
%
%
%

\section{First-order information complexity}


    
    
    
    
    
    
    




We consider the problem class of {\em mixed-integer convex optimization}:
\begin{equation}\label{eq:MICO}\inf\{f(\x,\y): (\x,\y) \in C, (\x,\y) \in \Z^n \times \R^d\},\end{equation} 
where $f:\R^n\times \R^d \to \R$ is a convex (possibly nondifferentiable) function and $C\subseteq \R^n\times \R^d$ is a closed, convex set. 
Given $\epsilon \geq 0$, we wish to report a point $(\x,\y)\in C \cap (\Z^n\times \R^d)$ such that $f(\x,\y) \leq f(\x', \y') + \epsilon$ for all $(\x',\y') \in C \cap (\Z^n\times \R^d)$. 
%
Such a point will be called an {\em $\epsilon$-approximate solution} and points in $C \cap (\Z^n\times \R^d)$ will be called {\em feasible solutions}. We say that $\x_1, \ldots, \x_n$ are the {\em integer-valued decision variables} or simply the {\em integer variables} of the problem, and $\y_1, \ldots, \y_d$ are called the {\em continuous variables}.

The notion of {\em information complexity} (a.k.a. {\em oracle complexity} or {\em analytical complexity}) goes back to foundational work by Nemirovski and Yudin~\cite{Nemirovski_Yudin_book} on convex optimization (without integer variables) and is based on the following. An algorithm for reporting an $\epsilon$-approximate solution to an instance $(f,C)$ must be ``given" the instance somehow. Allowing only instances with explicit, algebraic descriptions (e.g., the case of linear programming) can be restrictive in some settings. To work with more general, nonlinear instances, the algorithm is allowed to make queries to an oracle to collect information about the instance. 
More formally, we have the following definition.
\medskip

\begin{definition}
    An oracle $\cO$ for an optimization problem class $\cI$ is given by a family $\mathcal Q$ of possible queries
along with a set $H$ of possible answers or responses. A query $q \in \mathcal{Q}$ is a function $q : \cI \rightarrow H$. We say that $q(I) \in H$ is the
{\em answer} or {\em response} to the query $q$ for the instance $I \in \cI$.
\end{definition}
\medskip

Any algorithm using such an oracle to find an $\eps$-approximate solution for an instance 
makes queries about the instance in a sequence according to some strategy depending on the queries made and answers received, which we define formally as its {\em query strategy}. 
\medskip

\begin{definition}\label{def:query-strategy}
A {\em query strategy} is a function $D: (\mathcal{Q}\times H)^* \to \mathcal{Q}$, where $(\mathcal{Q}\times H)^*$ denotes the set of all finite sequences over $\mathcal{Q}\times H$, including the empty sequence. The {\em transcript $\Pi(D,I)$ of a strategy $D$ on an instance $I = (f,C)$} is the sequence of query and response pairs $(q_i, q_i(I))$, $i=1,2,\ldots$ obtained when one applies $D$ on $I$, i.e., $q_1 = D(\emptyset)$ and $q_i = D((q_1, q_1(I)), \ldots, (q_{i-1},q_{i-1}(I)))$ for $i \geq 2$.
\end{definition}
\medskip

If different instances with no common $\eps$-approximate solution produce the same transcript for the queries an algorithm has made, then the algorithm cannot tell them apart and will be unable to reliably report an $\eps$-solution for those instances after those queries. The goal is to design a query strategy that can report an $\epsilon$-approximate solution after making the smallest number of queries. This motivates the following definition of information complexity:
\medskip
\begin{definition} \label{def: icomp} Given a family of instances $\cI$ and access to an oracle $\cO$, 
the {\em $\epsilon$-information complexity $\icomp\textstyle{_\epsilon}(D, I, \cO)$ of an instance $I$ for a query strategy $D$}, 
is defined as the minimum natural number $k$ such that the set of all instances in $\cI$ for which $\cO$ returns the same responses as the instance $I$ to the first $k$ queries of $D$ have a common $\epsilon$-approximate solution. 
The {\em $\epsilon$-information complexity of the problem class $\cI$ with respect to the oracle $\cO$},  
is defined as $$\icomp\textstyle{_\epsilon}(\cI,\cO) := 
\inf_{D}\sup_{I\in \cI}\icomp\textstyle{_\epsilon}(D,I,\cO)
$$ where the infimum is taken over all query strategies.
\end{definition}

Thus, to prove an upper bound $u$ on $\icomp\textstyle{_\epsilon}(\I,\cO)$, it suffices to construct a query strategy that requires, in the worst case, at most $u$ queries to narrow down to a collection of instances that all have a common $\eps$-approximate solution. On the other hand, to establish a lower bound of $\ell$ on $\icomp\textstyle{_\epsilon}(\I,\cO)$, one needs to show that for any query strategy $D$, there exists a collection of instances in $\I$ that give the same responses to the first $\ell$ queries of $D$ (on these instances), and there is no point in $\R^n \times \R^d$ that is a common $\epsilon$-approximate solution to all these instances. 

While we introduce information complexity allowing for any general choice of oracle, the standard oracle that has been studied over the past several decades for convex optimization is the so-called  {\em (full-information) first-order oracle}, which has two types of queries indexed by points in $\R^n\times \R^d$: i) a {\em separation oracle} query indexed by a point $\z \in \R^{n + d}$ reports ``YES" if $\z \in C$ and otherwise reports a separating hyperplane for $\z$ and $C$, ii) a {\em subgradient oracle} query indexed by a point $\z \in \R^{n + d}$ reports $f(\z)$ and a subgradient for $f$ at $\z$. 
Tight lower and upper bounds (differing by only a small constant factor) on the number of queries required were obtained by Nemirovski and Yudin in their seminal work~\cite{Nemirovski_Yudin_book} for the case with no integer variables; roughly speaking, the bound is $\Theta\left(d\log\left(\frac{1}{\epsilon}\right)\right)$. These insights were extended to the mixed-integer setting in~\cite{oertel2014integer,basu2017centerpoints,basu2021complexity}, with the best known lower and upper bounds stated in~\cite{basu2021complexity}. 

Observe that the response to any separation/subgradient query is a vector in $\R^{n+d}$. Thus, each query reveals at least $n+d$ bits of information about the instance. A more careful accounting that measures the ``amount of information" accrued would track the total number of bits of information obtained as opposed to just the total number of oracle queries made. 
A natural question, posed in~\cite{basu2021complexity}, is whether the bounds from the classical analysis would change if one uses this new measure of the total number of bits, as opposed to the number of queries. The intuition, roughly, is that one should need a factor $(n+d)\log\left(\frac{1}{\epsilon}\right)$ larger than the number of first-order queries, because one should need to probe at least $\log\left(\frac{1}{\epsilon}\right)$ bits in $n+d$ coordinates to recover the full subgradient/separating hyperplane (up to desired approximations). We attempt to make some progress on this question in this paper. 

The above discussion suggests that one should consider oracles that return a desired bit of a desired coordinate of the separating hyperplane vector or subgradient. However, one can imagine making other binary queries on the instance; for example, one can pick a direction and ask for the sign of the inner product of the subgradient and this direction. In fact, one can consider more general binary queries that have nothing to do with subgradients/separating hyperplanes. If one allows {\em all} possible binary queries, i.e., one can use any function from the space of instances to $\{0,1\}$ as a query, then one can simply ask for the appropriate bits of the true minimizer and in $O((n+d)\log(1/\epsilon))$ queries, one can get an $\epsilon$-approximate solution. A matching lower bound follows from a fairly straightforward counting argument. Thus, allowing for all possible binary queries gives the same information complexity bound as the original Nemirovski-Yudin bound with subgradient queries in the $n=0$ (no integer variables) case, but is an exponential improvement when $n\geq 1$ (see~\cite{basu2021complexity} and the discussion below). What this shows is that the bounds on information complexity can be quite different under different oracles. With all possible binary queries, while each query reveals only a single bit of information, the queries themselves are a much richer class and this compensates to give the same bound in the continuous case and exponentially better bounds in the presence of integer variables. Thus, to get a better understanding of this trade-off, we restrict to queries that still extract information by only acting ``locally''.



\subsection{Our contributions}

\paragraph{Oracles based on first-order information.} Our first contribution is formalizing this notion of general ``local'' queries. While we focus on first-order information, our framework can be readily extended to consider, for example, information from higher-order derivatives.


\begin{definition}\label{def:SBO} An {\em oracle using first-order information}  $\cO(\cG, \cH)$ consists of two parts:
\begin{enumerate}
    \item For every $\z \in [-R,R]^{n + d}$, there exist three maps $\g^{\sepp}_\z: \I_{n,d,R,\rho,M} \to \R^{n + d}$, $\g^{\val}_\z: \I_{n,d,R,\rho,M} \to \R$, and $\g^{\sub}_\z: \I_{n,d,R,\rho,M} \to  \R^{n + d}$ such that for all $(f,C) \in \I_{n,d,R,\rho,M}$ the following properties hold. 
    \begin{enumerate}
        \item $C \subseteq \{\z' \in \R^{n + d}: \langle \g^{\sepp}_\z(f, C), \z' \rangle < \langle \g^{\sepp}_\z(f, C), \z \rangle\}$ if $\z \not\in C$ and $\g^{\sepp}_\z(f, C)=\0$ if $\z \in C$. In other words, $\g^{\sepp}_\z(f, C)$ returns a (normal vector to a) separating hyperplane if $\z \not\in C$. We will assume that a nonzero response $\g^{\sepp}_\z(f, C)$ has norm 1, since scalings do not change the separation property.
        \item $\g^{\val}_\z(f, C) = f(\z)$. In other words, $\g^{\val}_\z(f, C)$ returns the function value for $f$ at $\z$.
        \item $\g^{\sub}_\z(f, C) \in  \partial f(\z)$, where $\partial f(\z)$ denotes the subdifferential (the set of all subgradients) of $f$ at $\z$. In other words, $\g^{\sub}_\z(f, C)$ returns a subgradient for $f$ at $\z$. 
    \end{enumerate} 
    
    \noindent Such maps will be called {\em first-order maps}. A collection of first-order maps, one for every $\z$, is called a {\em first-order chart} and will be denoted by $\mathcal{G}$.
    \item There are three sets of functions $\mathcal{H}^{\sepp}$, $\mathcal{H}^{\val}$, and $\mathcal{H}^{\sub}$ and  with domains $\R^{n + d}$, $\R$ and $\R^{n + d}$ respectively. We will use the notation $\mathcal{H} = \mathcal{H}^{\sepp} \cup \mathcal{H}^{\val}\cup \mathcal{H}^{\sub}$. $\mathcal{H}$ will be called the collection of {\em permissible queries of the oracle}.
\end{enumerate}
\end{definition}

An algorithm for instances of~\eqref{eq:MICO} using $\cO(\cG, \cH)$ can, at any iteration, choose a point $\z$ and a function $h \in \mathcal{H}$ and receive the response $h(\g^{\sepp}_\z(\widehat f, \widehat C))$, $h(\g^{\val}_\z(\widehat f, \widehat C))$ or $h(\g^{\sub}_\z(\widehat f, \widehat C)$, depending on whether $h \in \mathcal{H}^{\sepp}$, $h \in \mathcal{H}^{\val}$ or $h\in \mathcal{H}^{\sub}$, where $\widehat f$ and  $\widehat C$ are the objective function and feasible region, respectively, of the unknown instance. Hence, queries to an oracle $\cO(\cG, \cH)$ using first-order information are indexed by $(\z, h)$, $\z \in \R^{n} \times \R^{d}, h\in \cH$. Since the goal of this paper is to provide bounds for different types of such oracles, i.e., with different permissible queries $\cH$, let us define some cases of interest.
\medskip

\begin{definition}[Examples of oracles] \label{def:oracle-example}
~

\begin{enumerate}
    \item (Full-information first-order oracle) When $\mathcal{H}$ consists only of the identity functions, i.e.,  $h^{\sepp}(\g_\z^\sepp(\widehat f, \widehat C)) = \g_\z^\sepp(\widehat f, \widehat C)$, $h^\val(\g_\z^\val(\widehat f, \widehat C)) = \g_\z^\val(\widehat f, \widehat C)$ and $h^\sub(\g_\z^\sub(\widehat f, \widehat C)) = \g_\z^\sub(\widehat f, \widehat C)$, we recover a full-information first-order oracle.
    
    \item (Bit oracle) Let $\mathcal{H}^{\textrm{bit}}$ be the set of binary queries that return a desired bit (of a desired coordinate) of the binary representation of $\g_\z^\sepp(\widehat f, \widehat C)$, $\g_\z^\val(\widehat f, \widehat C)$ or $\g_\z^\sub(\widehat f, \widehat C)$. Let $\mathcal{H}^{\textrm{bit}^*}$ be the {\em shifted} bit oracle that additionally returns a desired bit of $\g_\z^\val(\widehat f, \widehat C)+u$, for any $u \in \R$, i.e. $\mathcal{H}^{\textrm{bit}^*}$ allows querying a bit of the function value shifted by some number. 
    
    \item (Inner product threshold queries) Let \begin{align*}
        \mathcal{H}^{\textrm{dir}}&:= \{h^{\sepp}_{\u, c}: h^{\sepp}_{\u, c}(\g_\z^\sepp(\widehat f, \widehat C)) = sgn (\langle \u, \g_\z^\sepp(\widehat f, \widehat C) \rangle - c), \u \in \R^{n+d}, c\in \R\}
        \\ &\cup \{h^\val_{u, c}: h^\val_{u, c}(\g_\z^\val(\widehat f, \widehat C)) = sgn (u\cdot  \g_\z^\val(\widehat f, \widehat C) -c), u \in \R, c\in \R\}
        \\ &\cup \{h^\sub_{\u, c}: h^\sub_{\u, c}(\g_\z^\sub(\widehat f, \widehat C)) = sgn (\langle \u, \g_\z^\sub(\widehat f, \widehat C) \rangle-c), \u \in \R^{n+d}, c\in \R\},
    \end{align*} where \textit{sgn} denotes the sign function, be the set of binary queries that answers whether the inner product of the separating hyperplane, function value or subgradient, with a vector or a number of choice $\u$ or $u$ in the appropriate space, is at least some value $c$ or not. We write these as $\mathcal{H}^{\textrm{dir}}$ since these queries allow for the choice of a ``direction" $\u$, or a number $u$ in the function value case, as part of the query.
    
    \item When $\cH$ is the the set of all possible binary functions on $\R^n\times \R^d$ for the separating hyperplanes, $\R$ for the functions values, and $\R^n \times \R^d$ for the subgradients, we will call the resulting oracle the {\em general binary oracle based on $\cG$.} 
\end{enumerate} 
\end{definition}

These now give us a variety of oracles using first-order information, that clearly provide very different information for each query depending on the choice of permissible queries $\cH$. Note that different first-order charts will result in different oracles of each of these types that may give different answers at any point, depending on which separating hyperplane/subgradient the oracle's first-order map selects at those points for that instance.
\medskip 

We are now ready to state our quantitative results for lower and upper bounds on the information complexity of mixed-integer convex optimization under different oracles; see Table \ref{table:results} for a summary. It is not hard to see that we need to restrict the set of possible instances $\cI$ in order to have meaningful (finite) information complexity $\icomp_{\epsilon}(\cI, \cO)$. We will focus on the following standard parameterization. 

\begin{definition}\label{def:MICO-param} Define $\I_{n,d,R,\rho,M}$ to be the set of all instances of~\eqref{eq:MICO} such that:

\begin{enumerate}
\item[(i)] $C$ is contained in the box $\{\z \in \R^n \times \R^d: \|\z \|_\infty \leq R\}$. The case $C = \{\z \in \R^n \times \R^d: \|\z \|_\infty \leq R\}$ will be called {\em unconstrained}. 
\item[(ii)] If $(\x^\star, \y^\star)$ is an optimal solution of the instance, then there exists $\hat \y \in \R^d$ satisfying $\{(\x^\star,\y): \|\y - \hat \y\|_\infty \leq \rho\} \subseteq C$. In other words, there is a ``strictly feasible" point $(\x^\star, \hat \y)$ in the same fiber as the optimum $(\x^\star, \y^\star)$. 
\item[(iii)] $f$ is Lipschitz continuous with respect to the $\| \cdot \|_\infty$-norm with Lipschitz constant $M$ on $\{\x\} \times [-R,R]^d$ for all $\x \in [-R,R]^n \cap \Z^n$.
In other words, for any $(\x, \y), (\x, \y') \in (\Z^n\times \R^d) \cap [-R,R]^{n + d}$ with $\|\y - \y'\|_\infty \leq R$, $|f(\x,\y) - f(\x,\y')|\leq M\|\y - \y'\|_\infty$ with the convention that $\infty - \infty = 0$.
\end{enumerate} 
\end{definition}

\begin{table}[h!]
\centering
\footnotesize
\begin{tabular}[t]{p{2.5cm}p{1.6cm}p{4.7cm}p{5.2cm}}
\toprule
\raggedright Type of first-order oracle $\cO(\cG, \cH)$ & 
Variables  & Lower bound & Upper bound \\
\midrule
$\cH$ is hereditary         
& Mixed  & $\Omega(2^n\ell)$, where \newline ${\ell \leq \icomp_\eps(\cI_{0, d, R, \rho, M}, \cO(\cG, \cH))}$ \newline (Theorem \ref{thm: transfer theorem pure optimization}) \newline &                       \\
Full-information first-order oracle 
& Mixed & $\Omega\left(2^nd\log\left(\frac{MR}{\min\{\rho,1\}\epsilon}\right)\right)$ \newline (Corollary \ref{cor:standard-LB})    \newline                                                                                              & $O\left(2^nd (n+d) \log\left(\frac{MR}{\min\{\rho,1\}\epsilon}\right)\right)$ \newline (Oertel~\cite{oertel2014integer}, Basu-Oertel~\cite{basu2017centerpoints})            \\
\raggedright $\cH^{bit}, \cH^{bit^*}, \cH^{dir}$, or General Binary Queries                            

& Mixed      & $\tilde\Omega\left(2^n\max\left\{d^{\frac{8}{7}}, d\log\left(\frac{MR}{\min\{\rho,1\}\epsilon}\right)\right\}\right)$ \newline (Theorem \ref{thm:binary-LB})                                                           & $ O\left(2^n d\, (n+d)^2 \log^2\left(\frac{(n+d) MR}{\min\{\rho,1\}\epsilon}\right)\right)$ \newline (Theorem \ref{thm:binary-UB-mixed})  \\


& Continuous & $\tilde\Omega\left(\max\left\{d^{\frac{8}{7}}, d\log\left(\frac{MR}{\min\{\rho,1\}\epsilon}\right)\right\}\right)$ \newline (Theorem \ref{thm:binary-LB}) \newline  & $O\left(d^2 \log^2\left(\frac{d MR}{\min\{\rho,1\}\epsilon}\right)\right)$ \newline (Theorem \ref{thm:binary-UB-cont})  \\
General Binary 
& Mixed  & $\tilde\Omega\left(2^n\max\left\{d^{\frac{8}{7}}, d\log\left(\frac{MR}{\min\{\rho,1\}\epsilon}\right)\right\}\right)$ \newline (Theorem \ref{thm:binary-LB}) \newline  & $O\left(\log |\cI| + 2^nd(n+d)\log\left(\frac{MR}{\min\{\rho, 1\}\eps}\right)\right)$                        
\newline (Corollary \ref{cor:binary-UB-finite})\\

 & Continuous  & $\tilde\Omega\left(\max\left\{d^{\frac{8}{7}}, d\log\left(\frac{MR}{\min\{\rho,1\}\epsilon}\right)\right\}\right)$ \newline (Theorem \ref{thm:binary-LB})  & $O\left(\log |\cI| + d\log\left(\frac{MR}{\min\{\rho, 1\}\eps}\right)\right)$                        
\newline (Corollary \ref{cor:binary-UB-finite})
\\ \bottomrule
\end{tabular}
\caption{Summary of results on the information complexity of mixed-integer convex optimization for the class of instances $\I_{n,d,R,\rho,M}$ that have $n$ integer variables, $d$ continuous variables, the feasible region lies in the box $[-R,R]^{n+d}$ and has a ``$\rho$-deep feasible point'' on the optimal fiber, and the objective function is $M$-Lipschitz with respect to $\ell_{\infty}$ (see Definition~\ref{def:MICO-param}). The table presents simplified bounds showing only the main parameters.} 
\label{table:results}
\end{table}


\paragraph{Lower bounds.} 

Our first result is a ``transfer'' theorem that will be a powerful tool for obtaining concrete mixed-integer lower bounds under different oracles. This theorem lifts lower bounds for unconstrained optimization from the continuous to the mixed-integer setting. In particular, if one has a lower bound $\ell$ with respect to an oracle using first-order information (Definition~\ref{def:SBO}) for the information complexity for some family of purely continuous instances, then one can ``transfer" that lower bound to the mixed-integer case as $\Omega(2^{n}\ell)$
with access to the ``same" oracle in the $n+d$ dimensional space. For this notion, we require the set of permissible queries $\cH$ to be {\em hereditary}. Roughly speaking, this means that the set of queries has the same richness on a purely continuous space as it is in a mixed-integer space. We formally define hereditary queries in Section \ref{sec: proof of transfer theorem pure optimization}, and note that all 
of the types of permissible queries discussed in Definition~\ref{def:oracle-example} satisfy this property, except for $\cH^{bit}$ (the slightly enhanced $\cH^{bit^*}$ queries are hereditary).

\begin{theorem}\label{thm: transfer theorem pure optimization} Let $\cH_{n,d}$ be any class of hereditary permissible queries, and assume $\cH_{0,d}$ contains function threshold queries $h_c$ that answer $h_c(\g_\z^{\val}(\widehat{f}, \widehat C)) := sgn(\g_\z^{\val}(\widehat{f}, \widehat C)+c)$  for any $c\in \R$. 
Let $\epsilon \geq 0$. 
Suppose, for some $d\geq 1$, there exists a class $\I\subseteq \I_{0,d,R,\rho,M}$ of continuous convex (unconstrained) optimization problems in $\R^d$, and a first order chart $\cG_0$ for $\I$ such that $\icomp_{\epsilon}(\I,\cO(\cG_0, \cH_{0,d})) \geq \ell$. Suppose further that all instances in $\I$ have the same optimal value. Then, for any number of integer variables $n \geq 1$, there is a first order chart $\cG_n$ such that $\icomp_{\epsilon}(\I_{n,d,R,\rho,M}, \cO(\cG_n, \cH_{n,d})) \geq 2^{n-1}\ell$.
%
%
%
\end{theorem}
\medskip 




As a first consequence of this transfer theorem we obtain a sharpened lower bound for the standard full-information first-order oracle case for mixed-integer problems. For this setting, Basu~\cite{basu2021complexity} proved the lower bound of $\Omega\big(2^n\cdot d \log\big(\frac{2R}{3\rho}\big)\big)$. However, this bound is independent of the Lipschitz constant $M$ of the objective function, and thus does not capture the hardness of the problem as $M$ increases. By applying Theorem \ref{thm: transfer theorem pure optimization} to the classical lower bound of $\Omega \big(d\log \big( \frac{MR}{\eps}\big)\big)$ for continuous convex optimization with the standard first-order oracle by Nemirovski and Yudin \cite{Nemirovski_Yudin_book}, and combining the result with the existing mixed-integer lower bound, we obtain the following improved bound.

\begin{corollary}\label{cor:standard-LB} There exists a first-order chart $\mathcal{G}$ such that for the full-information first-order oracle based on $\mathcal{G}$ (i.e., $\mathcal{H}$ consists of the identity functions) we have
$$\icomp\textstyle{_\epsilon}(\cI_{n,d,R,\rho,M}, \cO(\mathcal{G}, \mathcal{H})) = \Omega \left(2^n \left(1 + d \log\left(\frac{MR}{\min\{\rho,1\}\epsilon}\right)\right)\right).$$
\end{corollary}

Moving on to ``non-standard'' oracles, we consider mixed-integer convex optimization under the general binary oracle. Recall from Definition \ref{def:oracle-example} that this means that the algorithm can make any binary query on subgradients/separating hyperplanes. Despite the power of these queries, we prove a separation between the information complexity under the standard full-information first-order oracle and the general binary oracle, i.e., the latter provides quantitatively less information for solving the problem. For example, in the pure continuous setting, $O(d)$  queries suffice (ignoring the logarithmic dependence on other parameters) under the full-information first-order oracle. However, we show that $\Omega(d^{8/7})$ queries are needed under the general binary oracle. More precisely, we show the following lower bound.

\begin{theorem}\label{thm:binary-LB}
For every $n \geq 0$, there exists a first-order chart $\mathcal{G}$ such that for the general binary oracle based on $\mathcal{G}$ 
we have
$$\icomp\textstyle{_\epsilon}(\cI_{n,d,R,\rho,M}, \cO(\mathcal{G}, \mathcal{H})) = \tilde\Omega \left( 2^n \left(1+\max\left\{d^{\frac{8}{7}}, d \log\left(\frac{MR}{\min\{\rho,1\}\epsilon}\right)\right)\right\}\right),$$ where $\tilde \Omega$ hides polylogarithmic factors in $d$.
\end{theorem}

We note that, since $\cH^{bit}$, $\cH^{bit^*}$ and $\cH^{dir}$ are more restrictive than the general binary oracle, this lower bound applies to oracles with those permissible queries as well. The proof of this result relies on a connection between information complexity and \emph{memory constrained} algorithms for convex optimization,  and the recent lower bound for the latter from~\cite{marsden2022efficient} (in addition to Theorem \ref{thm: transfer theorem pure optimization} for lifting the result to the mixed-integer case).



\paragraph{Upper bounds.} We now present upper bound results that illustrate the connection between information complexity based on full-information first-order oracles and information complexity based on binary queries on separating hyperplanes and subgradients. We first
formalize the intuition that by making roughly $O\left((n+d)\log\left(\frac{1}{\epsilon}\right)\right)$ bit or inner product sign queries on a separating hyperplane or subgradient, one should have enough information to solve the problem as with full information (Theorems~\ref{thm:binary-UB-mixed} and~\ref{thm:binary-UB-cont}). Next, in Theorem~\ref{thm:binary-UB-finite-transfer} and Corollary~\ref{cor:binary-UB-finite}, we show how this natural bound can be improved in certain settings.
\begin{theorem}\label{thm:binary-UB-mixed} Assume $d \geq 1$. For $U > 0$, consider the subclass of instances of $\I_{n,d,R,\rho,M}$ whose objective function values lie in $[-U,U]$, and the fiber over the optimal solution contains a $\z$ such that the $(n+d)$-dim $\rho$-radius $\ell_\infty$ ball centered at $\z$ is contained in $C$. 
There exists a query strategy for this subclass that reports an $\epsilon$-approximate solution by making at most
%
$$O \left(2^n d\, (n+d) \log\left(\frac{MR}{\min\{\rho,1\}\epsilon}\right)\right) \cdot \left((n+d) \log\left(\frac{(n+d) MR}{\rho\epsilon}\right) + \log \frac{U}{\epsilon} \right)$$
queries to an oracle $\cO(\mathcal{G},\mathcal{H})$, where $\mathcal{G}$ is any first-order chart and $\mathcal{H}$ is either $\mathcal{H}^{\textup{bit}}$ or $\mathcal{H}^{\textup{dir}}$.
\end{theorem}


Prescribing an {\em a priori} range for objective function values is not a serious restriction for two reasons: i) The difference between the maximum and the minimum values of an objective function in $\I_{n,d,R,\rho,M}$ is at most $2MR$, and ii) All optimization problems whose objective functions differ by a constant are equivalent. We also comment that while we assume $d\geq 1$ in Theorem~\ref{thm:binary-UB-mixed}, similar bounds can be established for the $d=0$ (pure integer) case. We omit this here because a unified expression for the $d=0$ and $d\geq 1$ cases becomes unwieldy and difficult to parse.

The main idea behind Theorem~\ref{thm:binary-UB-mixed} is to show that existing methods with the best known information complexity for mixed-integer convex optimization that use full-information first-order oracles can also work with approximate separation and subgradient oracles that return desired approximations of the true vectors (with no loss in the information complexity). Then one shows that one can produce these approximations with roughly $O\left((n+d)\log\left(\frac{1}{\epsilon}\right)\right)$ bit or inner product sign queries on a separating hyperplane or subgradient. With bit queries, this is just a matter of probing enough bits of each coordinate of the vector. The case with inner product sign queries is a bit more involved and our main tool is a result that shows how to approximate any vector up to desired accuracy with such queries (Lemma~\ref{lemma:gradapprox}).


Subsequently, using similar techniques we present an enhanced upper bound for the scenario where $n=0$ (pure continuous case).

\begin{theorem}\label{thm:binary-UB-cont} For $U > 0$, consider the subclass of instances  of $\I_{n,d,R,\rho,M}$ where $n=0$ (pure continuous case) and the objective function values lie in $[-U,U]$. There exists a query strategy for this subclass that reports an $\epsilon$-approximate solution by making at most
%
$$O \left(d \log\left(\frac{MR}{\min\{\rho,1\}\epsilon}\right)\right) \cdot \left(d \log\left(\frac{d MR}{\rho\epsilon}\right) + \log \frac{U}{\epsilon} \right)$$
queries to an oracle $\cO(\mathcal{G},\mathcal{H})$, where $\mathcal{G}$ is any first-order chart and $\mathcal{H}$ is either $\mathcal{H}^{\textup{bit}}$ or $\mathcal{H}^{\textup{dir}}$. 
\end{theorem}

Finally, we provide a kind of transfer result that allows one to transfer algorithms designed for full-information first-order oracles to the (harder) setting of a general binary oracle.

\begin{theorem}\label{thm:binary-UB-finite-transfer} Suppose there exists an algorithm that reports an $\eps$-approximate solution for instances in $\cI_{n, d, \rho, M, R}$ with at most $u$ queries to the full-information first-order oracle based on a first-order chart $\cG$. Then, for any subclass of finitely many instances $\mathcal{I}\subset \mathcal{I}_{n,d, R, \rho, M, R}$,  there exists a query strategy for this subclass using the general binary oracle based on $\mathcal{G}$ that reports an $\epsilon$-approximate solution by making at most 
$$O\big(\log |\I| +u\big)
$$
queries. 
\end{theorem}
 
 Using the centerpoint-based algorithm from \cite{oertel2014integer,basu2017centerpoints}, we obtain the following corollary:

\begin{corollary}\label{cor:binary-UB-finite} 
    Given any subclass of finitely many instances $\mathcal{I}\subset \mathcal{I}_{n,d, R, \rho, M}$ and any first-order chart $\mathcal{G}$, there exists a query strategy for this subclass using the general binary oracle based on $\mathcal{G}$ that reports an $\epsilon$-approximate solution by making at most 
$$O\left(\log |\I| +2^n\,d\, (n+d) \log\left(\frac{dMR}{\min\{\rho,1\}\epsilon}\right)\right)
$$
queries. In the (pure continuous) case of $n=0$,
$O\left(\log |\I| +  d\log\left(\frac{MR}{\min\{\rho,1\}\epsilon}\right)\right)$ queries suffice.
\end{corollary}

In particular, when the number of instances under consideration is $|\cI| = O(2^{2^nd(n+d)})$, Corollary \ref{cor:binary-UB-finite} gives a strictly better upper bound than Theorem \ref{thm:binary-UB-mixed}. Similarly, for $n=0$, in the case when $|\cI| = O(2^{d})$, we get a better upper bound compared to Theorem \ref{thm:binary-UB-cont}; in fact, we beat the lower bound provided by Theorem \ref{thm:binary-LB}. This demonstrates that even with exponentially many instances under consideration, the case of finite instances yields a lower information complexity than the case of infinitely many instances.
We point out that the first-order chart $\cG$ must be known to implement the query strategy in Theorem \ref{thm:binary-UB-finite-transfer}. In contrast, the algorithms in Theorems~\ref{thm:binary-UB-mixed} and~\ref{thm:binary-UB-cont} are oblivious of the first order chart, i.e., they work with any first order chart.

\subsection{Discussion and future avenues} The concept of information complexity in continuous convex optimization and its study go back several decades, and it is considered a fundamental question in convex optimization. In comparison, much less work on information complexity has been carried out in the presence of integer constrained variables. Nevertheless, we believe there are important and challenging questions that come up in that domain that are worth studying. Further, even within the context of continuous convex optimization, the notion of information complexity has almost exclusively focused on the number of full-information first-order queries. As we hope to illustrate with the results of this paper, considering other kinds of oracles leads to very interesting questions at the intersection of mathematical optimization and information theory. In particular, the study of binary oracles promises to give a more refined understanding of the fundamental question ``How much information about an optimization instance do we need to be able to solve it with provable guarantees?". For instance, establishing {\em any superlinear (in the dimension)} lower bound for the continuous problem with binary oracles, like the one in  Theorem~\ref{thm:binary-LB}, seems to be nontrivial. In fact, the results from~\cite{marsden2022efficient}, on which Theorem~\ref{thm:binary-LB} is based, were considered a breakthrough in establishing superlinear lower bounds on space complexity of convex optimization. Even so, the right bound is conjectured to be quadratic in the dimension (see Theorem~\ref{thm:binary-UB-cont}) and our Theorem~\ref{thm:binary-LB} is far from that at this point. These other oracles also have a practical motivation. Obtaining exact first-order information may be difficult or impossible in many practical situations, and one has to work with approximations of separating hyperplanes and subgradients. The binary oracles can be viewed as providing these approximations and information complexity under these oracles becomes important from a practical standpoint. 

We thus view the results of this paper as expanding our understanding of information complexity of optimization in two different dimensions: what role does the presence of integer variables play and what role does the nature of the oracle play (with or without integer variables)? 
For the role of integer variables, in the pure optimization case Theorem \ref{thm: transfer theorem pure optimization} provides a lifting of lower bound from the continuous case. 
Allowing for constraints, Corollary \ref{cor:standard-LB} brings the lower bound closer to the best known upper bound on information complexity based on the classical subgradient oracle. The remaining gap is now simply a factor linear in the dimension. A conjecture in convex geometry first articulated in~\cite[Conjecture 4.1.20]{oertel2014integer} and elaborated upon in~\cite{basu2017centerpoints,basu2021complexity} would resolve this and would show that the right bound is essentially equal to the lower bound we prove in this paper. 

Beyond the contributions discussed above, our work also opens up new future directions for study. We believe the following additional conjectures to be good catalysts for future research, especially in regard to understanding the interplay of integer variables and other oracles.


The first conjecture is a generalization of our Theorem \ref{thm: transfer theorem pure optimization} to incorporate constraints as well. This would make this ``transfer" tool more powerful, and would, for example, give Corollary~\ref{cor:standard-LB} as a special case without appealing to \cite{basu2021complexity} for the feasibility lower bound. 

\begin{conj}\label{conj:transfer-mixed}
If there exist continuous, \textbf{constrained} convex optimization instances such that $\ell$ is a lower bound for this family on the information complexity with respect to an oracle, then for every $n\geq 1$, there exist mixed-integer instances with $n$ integer variables such that the information complexity of these mixed-integer instances is lower bounded by $\Omega(2^n\cdot\ell)$ for the same oracle.
\end{conj}

Another consequence of resolving this conjecture is that if future research on the information complexity of continuous convex optimization results in better/different lower bounds based on feasibility, these would immediately imply new lower bounds for the mixed-integer case. For instance, we believe the following conjecture to be true for the mixed-integer convex optimization problem.

\begin{conj}\label{conj:cont-binary-LB}
There exists a first-order chart $\mathcal{G}$ such that the general binary oracle based on $\mathcal{G}$ has information complexity $\Omega\left(2^n \Big(1+ d^2\log\left(\frac{MR}{\rho\epsilon}\right)\Big)\right)$.
\end{conj}

A version of Conjecture~\ref{conj:cont-binary-LB} is also stated in the language of ``memory-constrained" algorithms in~\cite{woodworth2019open,marsden2022efficient} for the continuous case (see Section~\ref{sec:information-mem} below); the way we have stated the conjecture here presents its transfer to the mixed-integer case.

Analogously, it would be nice to have ``transfer" theorems for upper bounds as well. In the spirit of Theorems~\ref{thm:binary-UB-mixed}, \ref{thm:binary-UB-cont} and~\ref{thm:binary-UB-finite-transfer}, we believe a useful result would be a theorem that
takes upper bound results proved in the full-information first-order oracle setting and obtains upper bound results in the general binary oracle setting. A use case of such a result would be the following: if the upper bound for the general mixed-integer problem with full-information first-order oracles is improved by resolving the convex geometry conjecture mentioned above (and we believe the lower bound is correct and the upper bound is indeed loose), then this would also give better upper bounds for the general binary oracle setting. Thus, we make the following conjecture.

\begin{conj}\label{conj:transfer-binary}
If there exists a query strategy with worst case information complexity \\ $u(n,d,R,\rho,M,\mathcal{G})$ under the full-information first-order oracle based on a first-order chart $\mathcal{G}$, then there exists a query strategy with worst case information complexity bounded by $$u(n,d,R,\rho,M,\mathcal{G})\cdot O\left((n+d)\log\left(\frac{MR}{\rho\epsilon}\right)\right)$$ under the general binary oracle based on $\mathcal{G}$.
\end{conj}

We focus on oracles that use first order information in this paper (Definitions~\ref{def:SBO} and~\ref{def:oracle-example}). Oracles that use ``zero-order information" have also been studied in the literature, beginning with the seminal work of Yudin and Nemirovski~\cite{Nemirovski_Yudin_book}; see~\cite{GroetschelLovaszSchrijver-Book88} for an exposition of how those ideas can be used in the mixed-integer setting and~\cite{conn2009introduction} for an exposition in the nonconvex setting. Such oracles report function values only for the objective, with no subgradient information, and only report membership for the constraints, with no separating hyperplanes. A related oracle is the ``value comparison" oracle that has found many applications. These oracles comprise of questions of the form ``Is $f(\z) \le f(\z')$?'', with no access to the subgradients of $f$. Such algorithms are  particularly useful in learning from users' behaviors, since while a user typically cannot accurately report its (dis)utility value $f(\z)$ for an option $\z$, it can more reliably compare the values $f(\z)$ and $f(\z')$ of two options; see~\cite{comparisonOptNIPS,comparisonOpt} and references therein for discussions and algorithms in the continuous convex case. The mixed-integer setting under the value comparison oracle has been extensively studied in recent work~\cite{chirkov2019complexity,gribanov2019integer,gribanov2020minimization,veselov2020polynomial}. The ideas in this paper can also be adapted to give algorithms for mixed-integer convex optimization using the comparison oracle, but we do not undertake a deeper study here. There seems to be scope for future research in this direction, especially in tying together these different strands of ideas for ``zero order information".

\bigskip
The remainder of this paper is dedicated to the formal proofs of our main results discussed above.


\section{Proof of Theorem \ref{thm: transfer theorem pure optimization}} \label{sec: proof of transfer theorem pure optimization}


The high-level idea for the proof of Theorem 7 is to construct difficult mixed-integer instances by taking hard instances of the continuous case, ``placing" one of them on each fiber $\x \times \R^d, \x \in \{0,1\}^n$, and interpolating between fibers appropriately. We do this in a way such that effectively one needs to solve the continuous problems obtained by restricting to each fiber, which leads the $\Omega(2^n \ell)$ lower bound from an $\ell$ lower bound on the continuous problems -- namely, there will be one difficult function from the continuous case placed on each of the $2^n$ fibers, so if one can't do better than solving each of them separately, one ends up with an $\Omega(2^n \ell)$ lower bound. To make this idea work, the interpolation needs to be done in a way that no query in the full $[0,1]^n\times \R^d$ space reveals information about two (or more) of the continuous functions placed on different fibers, or reveals significantly more information about a function on a fiber than a query on that fiber would. For example, we need to ensure that a single query at the point $(\frac12, \ldots, \frac12, \y)$ for $\y \in R^d$ does not reveal information about multiple functions on different fibers.



\subsection{Game-theoretic perspective} 

    So far we have described the information complexity of optimization using an oracle $\cO$ over a family of instances $\cI$ based on having an optimization algorithm that in each round $t$ makes a query $q_t$ to $\cO$ and receives as answer the result $q_t(\widehat{I})$ for the unknown instance $\hat{I}$ it is trying to optimize. However, for obtaining lower bounds on the information complexity, it is more helpful to consider the algorithm as interacting with an \emph{adversary} for the family of instances under $\cO$, instead of the unknown instance $\hat{I}$. More precisely, at round $t$, the adversary receives the query $q_t$ of the algorithm and produces, possibly based on all the previous queries $q_1,\ldots,q_{t-1}$, a response $r_t$. 
    The only requirement is that there must always exist at least one instance $\bar{I} \in \cI$ that is \emph{consistent} with all of its responses,  namely $r_{t} = q_t(\bar{I})$ for all $t$, under the oracle $\cO$ being considered. 
   With each such response, the set of instances that are consistent with all responses given may change, motivating the following definition: 
    \begin{definition}\label{def:surviving instances}
        Given a class of instances $\cI$, an oracle $\cO$, and a transcript of query-response pairs $(q_1, r_1),..., (q_t, r_t)$, the set of {\em surviving instances} for $(q_1, r_1),..., (q_t, r_t)$ under $\cO$ is
        $$\{ I \in \cI: q_j(I) = r_j \, \forall \, j \in [t]\},$$
        i.e., the set of instances consistent with the responses in the transcript under the oracle $\cO$. When all instances in $\cI$ are unconstrained, let the set of {\em surviving functions} be the set of functions corresponding to the surviving instances. 
    \end{definition}

  We say that an adversary \textbf{Adv} is \emph{$\eps$-hard for $\ell$ rounds}
     if for any algorithm \textbf{Alg}, after $\ell$ rounds there are surviving instances in $\cI$ that do not have a common $\eps$-approximate solution, i.e., if $q_1,\ldots,q_\ell$ and $r_1,\ldots,r_\ell$ are \textbf{Alg}'s queries and \textbf{Adv}'s responses, respectively, then there is a collection of instances $\mathcal{J} \subset \cI$ that have no common $\eps$-approximate solution but such that $r_t = q_t(I)$ for all $I\in \mathcal{J}$ and $t = 1,\ldots,\ell$. 
     Since the sets of $\eps$-approximate solutions of instances in $\cI_{n,d,R,\rho,M}$ are compact convex sets, this collection $\mathcal{J}$ of instances may always be taken to be finite\footnote{If a collection of compact sets has empty intersection, then there exists a finite subcollection that already has empty intersection.}.
     Intuitively, the existence of such an adversary should imply that no algorithm can reliably report an $\eps$-approximate solution within $\ell$ iterations, that is, $\icomp_\eps(\cO, \cI) > \ell$. The next result shows that this adversary-based perspective is indeed equivalent to information complexity, and may be of independent interest (for a proof see Appendix~\ref{sec: information games}).

    \begin{restatable}{lemma}{infoAdv}\label{lemma:infoAdv}
    Consider a class of instances $\I$ and an oracle $\cO$. Then $\icomp_\eps(\I, \cO) > \ell$ if and only if there exists an adversary under $\cO$ using $\cI$ that is $\eps$-hard for $\ell$ rounds.
    \end{restatable}


\subsection{Proof for the full-information first-order oracle}

The full proof of Theorem \ref{thm: transfer theorem pure optimization} is a bit technical and requires a few conceptual connections. For a better exposition, we first prove the theorem in the case that the oracle is the full-information first-order oracle. As $\cH$ thus consists of the identity maps, throughout this subsection we will write oracles using first-order information as $\cO(\cG)$, where $\cG$ is the corresponding first order chart.

Given the assumption of the theorem and the equivalent adversarial perspective from Lemma~\ref{lemma:infoAdv}, assume there is a family of continuous, unconstrained instances $\cI_{cont} \subseteq \cI_{0,d,R,\rho,M}$, all with the same optimal value $\OPT$, and a full-information first-order adversary \textbf{Adv-Cont} for $\cI$ that is $\eps$-hard for $\ell-1$ rounds.
Let us use $\F_{cont}$ to denote the objective functions of the instances $\cI_{cont}$.
In the full-information first-order case, queries of an optimization algorithm consist of points $\y_1,\y_2,\ldots \in \R^d$, and either query the function value or the subgradient. For simplicity, let us allow the algorithm to query \textit{both} the function value and subgradient in a single query, so that the queries become simply $\y_1,\y_2,\ldots \in \R^d$ and the responses of an adversary consist of a sequence of consistent function values and subgradients, namely a sequence $(v_1,\g_1), (v_2,\g_2),\ldots \in \R \times \R^d$ such that there is some $f \in \F_{cont}$ satisfying $v_t = f(\y_t)$ and $\g_t \in \partial f(\y_t)$ for all rounds $t$.

To prove the theorem, we will construct a full-information first-order adversary \textbf{Adv-MI} for a family of mixed-integer instances over $\{0,1\}^n \times \R^d$ that is $\eps$-hard for $2^n \ell - 1$ rounds. As alluded to before, the very high-level is to place a copy of the continuous adversary \textbf{Adv-Cont} on each of the continuous fibers $\x \times \R^d$ for $\x \in \{0,1\}^n$. In fact, we will work with a slightly modified version of the continuous adversary that is constructed next. 


\subsubsection{Modifying the continuous adversary \textbf{Adv-Cont}}


For the mixed-integer adversary \textbf{Adv-MI}, it will be important to render a fiber ``useless" for the optimization algorithm after it queries (close to) this fiber too many times, so as to intuitively force it to query (close to) other fibers, or gain no new information otherwise. This will be done by modifying the continuous adversary \textbf{Adv-Cont} such that whenever it is probed $\ell$ or more times, it commits to answering all future queries consistently with a \textit{single} function that has optimal value $> \OPT + \eps$; since our mixed-integer instances will be constructed to have optimal value $\OPT$, gathering more information about the function on such fibers will not help the algorithm solve the mixed-integer problem. To do this, the modified continuous adversary will also keep track of the set $S$ of surviving functions (Definition~\ref{def:surviving instances}) given its responses. More precisely, here are its main properties. 

\begin{lemma} \label{lemma:advContPlus}
    There is a family of convex functions $\xoverline{\mathcal{F}}_{cont}$ corresponding to instances $\cI_{0,d,R,\rho,M}$ of the purely continuous case, a first-order chart $\cG$ and a full-information first-order adversary \textbf{Adv-Cont+} that, for any algorithm \textbf{Alg}, maintains a set of functions $S_t \subseteq \xoverline{\mathcal{F}}_{cont}$ for every query-response round $t$  with the following properties:

    \begin{enumerate}
        \item In every round $t\geq 1$, all functions in $S_t$ are consistent with the responses returned by \textbf{Adv-Cont+} thus far, under some oracle using first-order information $\cO(\cG)$.

        \item \vspace{-6pt} In the first $t\leq \ell-1$ rounds, there is a finite collection of functions in $S_t \cap \mathcal{F}_{cont}$ that do not share an $\eps$-approximate solution. In particular, \textbf{Adv-Cont+} is still $\eps$-hard for $\ell-1$ rounds. 
            
        \item \vspace{-6pt} For all rounds $t\ge 1$, $S_t$ is closed under taking maxima of finitely many of its elements, and also contains a function that has minimum value $> \OPT + \eps$.

        \item For rounds $t\geq \ell$, $S_t$ contains a single function with minimum value $> \OPT+\eps$.

    \end{enumerate}
\end{lemma}


\medskip
Item 1.~means that for rounds $t < \ell$, there exists a full-information first-order oracle $\cO(\cG)$ such that $S_t$ is exactly the set of surviving functions under $\cO(\cG)$ given the responses produced by \textbf{Adv-Cont+} up to round $t$. 
Hence, we will refer to this $S_t$ as the set of \textit{surviving functions maintained by \textbf{Adv-Cont+} at round $t$}.
We now make precise our modification to the continuous adversary \textbf{Adv-Cont} and prove Lemma~\ref{lemma:advContPlus}. As a preliminary, let $\xoverline{\mathcal{F}}_{cont}$ denote the closure of $\mathcal{F}_{cont}$ under taking maxima of finitely many functions, i.e. for any finite collection $\mathcal{J} \subset \mathcal{F}_{cont}$, $\max_{f\in \mathcal{J}}\{f\} \in \xoverline{\mathcal{F}}_{cont}$. Notice these functions are still convex. The following lemma highlights the key property of $\mathcal{F}_{cont}$ we will make use:

\begin{lemma} \label{lemma:maxOPT}
    Let $\mathcal{J}\subset \mathcal{F}$ be a finite set. If $f\in \mathcal{J}$ do not have a common $\eps$-solution, then the pointwise maximium function $\max_{f\in \mathcal{J}}\{f\}$ has minimum value greater than $OPT + \eps$.   
\end{lemma}

\begin{proof}
Suppose for sake of contradiction that there exists a point $\z$ such that $\max_{f\in \mathcal{J}}\{f(\z)\} \leq OPT + \eps$. Then $f(\z) \leq OPT + \eps$ for all $f\in \mathcal{J}$, which means $\z$ is an $\eps$-solution for all $f\in \mathcal{J}$, which contradicts the assumption that they do not share an $\eps$-solution.
\end{proof}

We now formally describe \textbf{Adv-Cont+}, and then prove that it satisfies the invariants of Lemma~\ref{lemma:advContPlus}. 

    \vspace{4pt}
    \begin{mdframed}
    \vspace{-7pt}
    \begin{proc} \label{proc:advCont}
     \normalfont
     \textbf{Adv-Cont+} 

    \vspace{4pt}
    \noindent Initialize set of surviving functions $S_0 = \xoverline{\mathcal{F}}_{cont}$

    \vspace{4pt}
    \noindent For each round $t = 1,2\ldots$\,:

    \vspace{-8pt}
    \begin{enumerate}[leftmargin=18pt]
        \item Receive query point $\y_t \in \R^d$ from the optimization algorithm
        
        \item  \vspace{-3pt}  If $t \le \ell - 1$: Send $\y_t$ to the adversary \textbf{Adv-Cont}, receiving back a value $v_t$ and subgradient $\g_t$. Obtain $S_{t}$ by removing from $S_{t-1}$ the functions $f$ that are not consistent with this response for any first-order chart $\cG$, namely where $f(\y_t) \neq v_t$ or $\g_t \notin \partial f(\y_t)$.
        
        Send the response $(v_t,\g_t)$ to the optimization algorithm. 
       
        \item \vspace{-3pt} If $t = \ell$: Since \textbf{Adv-Cont} if $\eps$-hard for $\ell-1$ rounds, there is a finite collection of functions $\{f_1, ..., f_k\} \subset S_{t-1} \cap \mathcal{F}_{cont}$ that do not share an $\eps$-solution. Define their pointwise maxima $f_{\max} = \max\{f_1, ..., f_k\}$ and set $S_{t+k} = \{f_{\max}\}$, for all $k = 0, 1, 2...$. 
        
        Set the value $v_t$ to be $f_{\max}(\y_t)$ and set $\g_t$ to be a subgradient in $\partial f_{\max}(\y_t)$ (consistent with what the first order chart $\cG_0$ gives for $f_1, \ldots, f_k$ at $\y_t$, if $\y_t$ has been queried in an earlier round), and send the response $(v_t,\g_t)$ to the optimization algorithm. 
          
        \item \vspace{-3pt} If $t > \ell$: Let $f_{\max}$ be the only function in $S_{t-1}$. If $\y_t$ was queried in an earlier round $k$, answer $(v_k, \g_k)$. Otherwise, set the value $v_t$ to be $f_{\max}(\y_t)$ and set $\g_t$ to be any subgradient in $\partial f_{\max}(\y_t)$, and send the response $(v_t,\g_t)$ to the optimization algorithm. 
    \end{enumerate}
\end{proc}
\end{mdframed}
\smallskip
   
    \begin{proof}[Proof of Lemma \ref{lemma:advContPlus}]
     We will proceed by induction on the number of rounds $t$. The lemma clearly holds for $S_0$, so suppose it holds for $S_{t-1}$. 
     
    If $t \le \ell - 1$, then $S_t$ satisfies Item 1 due to to the update rule in the procedure for obtaining $S_t$, since all functions that are not consistent with the given response are removed. More precisely, since the responses given are those produced by \textbf{Adv-Cont}, these functions in $S_t$ are consistent with the responses under exactly the oracle $\cO(\cG_0)$ that \textbf{Adv-Cont} is hard under. 
    $S_t$ also satisfies Item~2 because \textbf{Adv-Cont} is assumed to be $\eps$-hard for $\ell-1$ rounds, so there exists a finite collection of functions $\{f_1, ..., f_k\} \subset \F_{cont}$ with no common $\eps$-solution that are consistent with all responses given to \textbf{Adv-Cont+} by \textbf{Adv-Cont}; thus $S_t$ contains them. 
    For Item 3, to show the closure under taking maxima, we need to argue that if functions $f_1, ..., f_k$ were not removed from $S$, then neither was $\max(f_1, ..., f_k)$. 
    Since $f_1, ..., f_k$ are convex, then $\partial f_j (\y) \subset \partial \max\{f_1, ..., f_k\}(\y)$ for any $j$ such that $f_j(\y) = \max\{f_1(\y), ..., f_k(\y)\}$.
    Hence, if $f_1, ..., f_k$ all have function value $v_t$ and subgradient $\g_t$ at $\y$, then so does $\max\{f_1, ..., f_k\}$, so $\max\{f_1, ..., f_k\}$ was not removed from $S$, as desired. Furthermore, if $f_1, ..., f_k$ are taken to be the functions guaranteed by Item 2, Lemma \ref{lemma:maxOPT} implies that $\max\{f_1, ..., f_k\}$ has optimal value greater than $\OPT+\eps$, so since we just showed $\max\{f_1, ..., f_k\} \in S_t$, the remainder of item 3 follows. 

    If $t = \ell$, $S_t$ contains the single function $f_{\max}$, which has optimal value greater than $\OPT+\eps$ by its construction as a consequence of Lemma \ref{lemma:maxOPT}. Hence, Item 3 follows. 
    To prove item 1, we will use that $S_{t-1}$ satisfies Item 1, and by Item 3 applied to $S_{t-1}$, $f_{\max}$ is consistent with the responses returned by the procedure up to round $t-1$. For round $t$ itself, consistency follows from the definition of $v_t$ and $\g_t$, and so $f_{\max}$ is consistent with all responses given. Item 2 does not apply in this case and Item 4 is immediate by the construction of $S_t := \{f_{\max}\}$.

    If $t > \ell$, then $S^{t} = S^{t-1} = \{f_{\max}\}$ and it suffices to check that the response $(v_t,\g_t)$ is compatible with $f_{\max}$, which follows immediately from the definition of the response.

   Hence, Items 1-4 of the lemma follow. It remains to show that \textbf{Adv-Cont+} is indeed a well-defined adversary under a full-information first-order oracle. For rounds $t\leq \ell-1$, this is inherited from \textbf{Adv-Cont}, while for $t \geq \ell$, this is ensured because if the queried point $\y_t$ is the same as  $\y_{t'}$ for some round $t'< t$, \textbf{Adv-Cont+} provides the same response in round $t$ as in round $t'$. Thus, there is indeed a first-order chart $\cG$ (derived from the first order map $\cG_0$ for \textbf{Adv-Cont}) such that \textbf{Adv-Cont+} is an adversary under the corresponding full-information first-order oracle $\cO(\cG)$.
%
%
%
%
    \end{proof}


\subsubsection{Constructing the mixed-integer adversary \textbf{Adv-MI}} \label{sec:MIAdv}
    We now construct the family $\F_{MI}$ of functions over $\R^n \times \R^d$ used to transfer the lower bound to the mixed-integer setting, along with the adversary \textbf{Adv-MI} for that family. We call functions over $\R^n \times \R^d$ \textit{full-dimensional} to distinguish them from the functions over $\R^d$, the continuous part of the problem. 
    As indicated previously, these full-dimensional functions $\psi$ in $\F_{MI}$ will be obtained by considering combinations of selecting one function $f_{\bar{\x}}$ from $\xoverline{\F}_{cont}$ for each of the mixed-integer fibers $\bar{\x} \times \R^d, \bar{\x} \in \{0,1\}^n$, letting $\psi$ equal the appropriate function selected over each corresponding fiber, 
    and applying an interpolation scheme between the fibers. This interpolation is illustrated in Figure \ref{fig:hardFuncMI} and described in detail later in this section. 
    
    For the behavior of \textbf{Adv-MI}, we instantiate a copy of the modified continuous adversary \textbf{Adv-Cont+} on each fiber. 
    Whenever the optimization algorithm queries a point $(\bar{\x}, \y)$ on a fiber, we send $\y$ to the continuous adversary on the fiber and report back the response $(v,\g)$ received, although $\g$ needs to be appropriately lifted to the full $\R^n \times \R^d$ space to be consistent with the way we interpolate the functions between the fibers. 
    If the optimization algorithm only probes on these fibers, then it is intuitive that such an adversary would be $\eps$-hard for $2^n \ell - 1$ rounds:
    informally, up to this round, at least one of the $2^n$ fibers that has received no more than $\ell - 1$ queries, so using the hardness of \textbf{Adv-Cont+} (Item 2 of Lemma \ref{lemma:advContPlus}) we can obtain full-dimensional functions that do not share an $\eps$-approximate solution, which confirm the desired $\eps$-hardness of the mixed integer adversary \textbf{Adv-MI}.

    The crucial element is how to deal with queries on points outside of the mixed-integer fibers. If such queries provide the algorithm with more information about the full-dimensional functions $\F_{MI}$ than queries on the fibers do, then we may not have full-dimensional functions with no common $\eps$-approximate solution surviving for $2^n \ell - 1$ rounds. To handle this issue, the interpolation used to define the full-dimensional functions $\psi$ guarantees that its behavior on a fractional point $(\tilde{\x}, \y) \notin \{0,1\}^n \times \R^d$ is completely determined by the value of the function $f_{\bar{\x}}(\y)$ from $\xoverline{\F}_{cont}$ selected for the fiber $\bar{\x}\times \R^d$, where $\bar{\x} \in \{0,1\}^n$ is the closest 0/1 point to $\tilde{\x}$. Thus, \textbf{Adv-MI} can also answer such a query at a fractional point by making a query to the appropriate continuous adversary \textbf{Adv-Cont+} on $\left\{\bar{\x}\right\}\times \R^d$, and the hardness of the latter can still be leveraged. 

    We now formally define the functions $\F_{MI}$ and the adversary \textbf{Adv-MI}.

\begin{figure}[h]
\centering
\includegraphics[height=6cm]{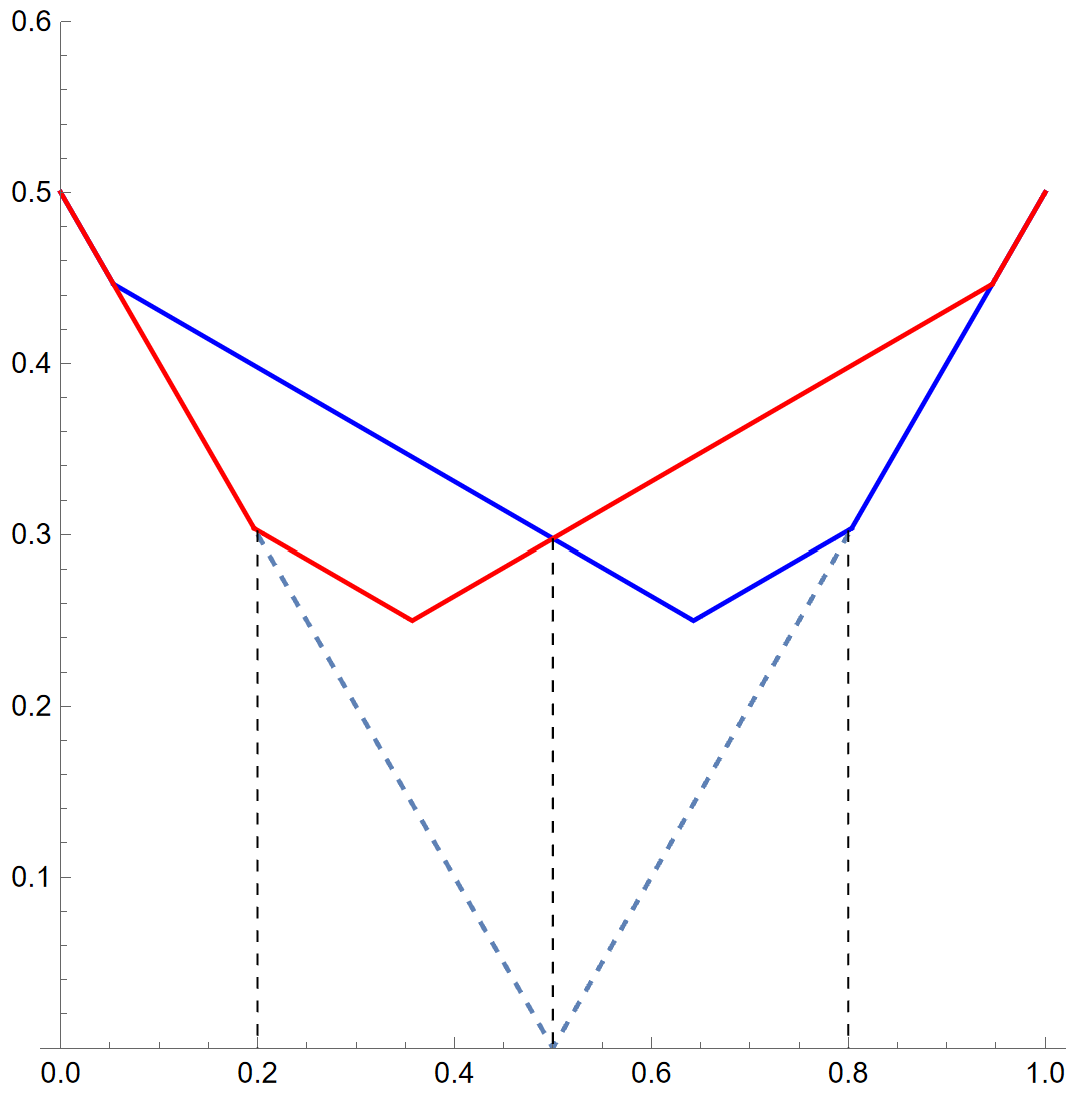}
\hspace{1cm}
\includegraphics[height=6cm]{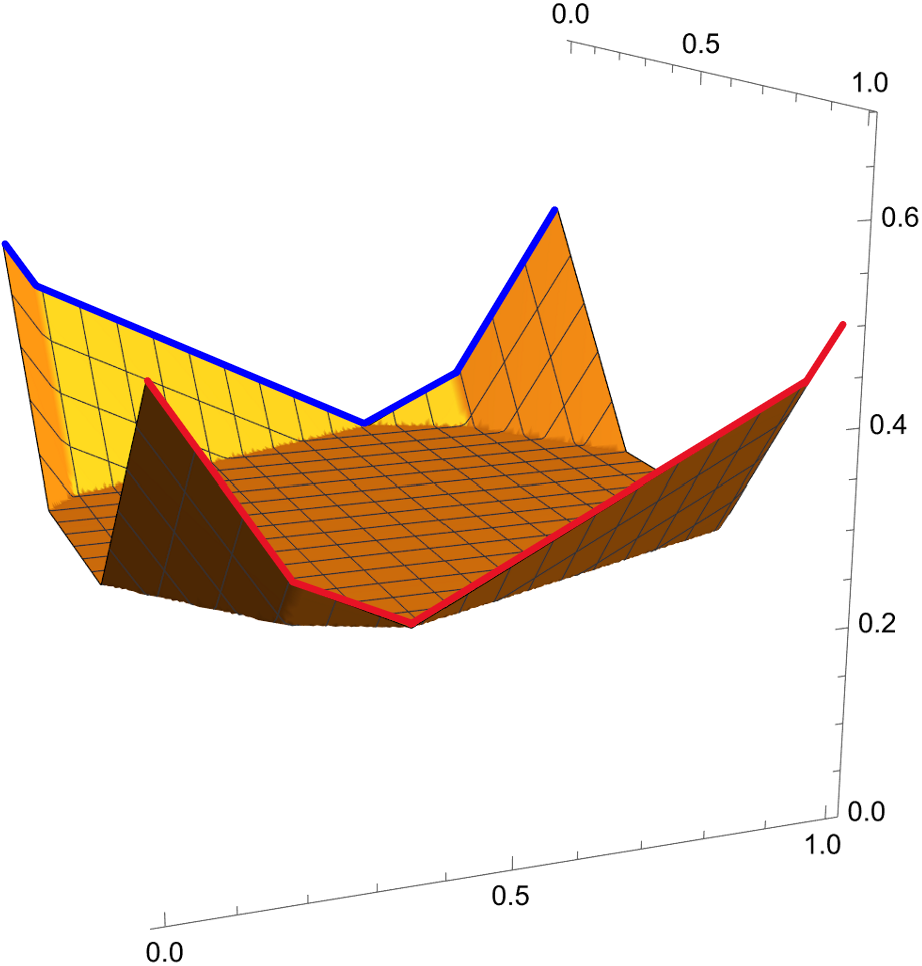}
\caption{(Left) Illustration of two possible functions $f_0,f_1 \in \xoverline{\F}_{cont}$ (blue and red) of the continuous adversary \textbf{Adv-Cont+}, for $d=1$. (Right) Illustration of the function $\psi_{(f_0,f_1)}$ for the mixed-integer adversary \textbf{Adv-MI} obtained by placing the functions $f_0$ and $f_1$ on the fibers $\{0\} \times \R$ and $\{1\} \times \R$ and interpolating appropriately between the fibers.
}
  \label{fig:hardFuncMI}
\end{figure}

\paragraph{Construction of the functions $\F_{MI}$.} For a 0/1 point $\bar{\x} \in \{0,1\}^n$ and a function $f \in \xoverline{\mathcal{F}}_{cont}$ in $\R^d$, 
we first define its (convex) extension to the full-dimensional space $\R^{n+d}$ as 
    \begin{align}
        \hat{f}_{\bar{\x}}(\x,\y) = \max\Big\{f(\y)  + \langle \mathbf{M}_{\bar{\x}}, \x - \bar{\x}\rangle~,~\OPT \Big\}, \label{eq:funcExt}
    \end{align}
    with $\mathbf{M}_{\bar{\x}} := 3MR \cdot sgn(\bar{\x}-0.5\cdot\ones),$ where $sgn$ denotes the sign function. This construction effectively places $f $ along the $\y$ space at the fiber $\bar{\x}$ and extends it in each of the $\x$ variables via a linear function with slope $\pm 3MR$, in a way that it decreases the value as it moves into the unit cube, or equivalently, away from $\bar{\x}$; it then truncates the final value to being at least $\OPT$; see Figure \ref{fig:single_function_extension} for an illustration.
    We note for later use that wherever the extension is not truncated by $\OPT$, a subgradient is given by appending the vector $\mathbf{M}_{\bar{\x}}$ to a subgradient of $f$, and otherwise the all zeroes vector is a subgradient. More precisely, we have  
    \begin{align}
        \partial \hat{f}_{\bar{\x}}(\x,\y) \supseteq \left\{\begin{array}{ll}
            \{\mathbf{M}_{\bar{\x}}\} \times \partial f(\y) & \textrm{, if $\hat{f}_{\bar{\x}}(\x,\y) > \OPT$} \\
            \{\zero\} & \textrm{, otherwise.}
        \end{array}\right.   \label{eq:subgradFHat}
    \end{align}

\begin{figure}[h]
\centering
\includegraphics[height=6cm]{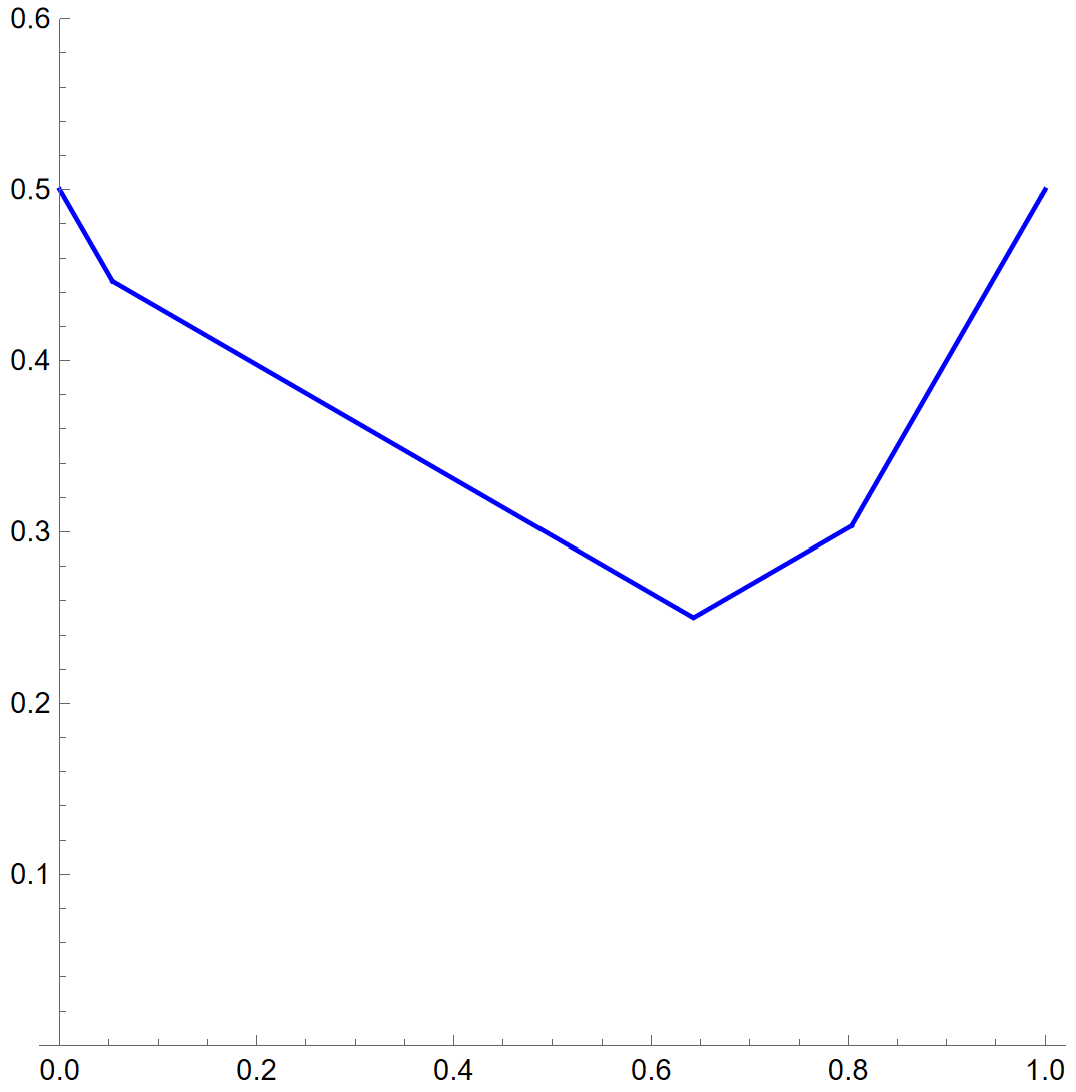}
\hspace{1cm}
\includegraphics[height=6cm]{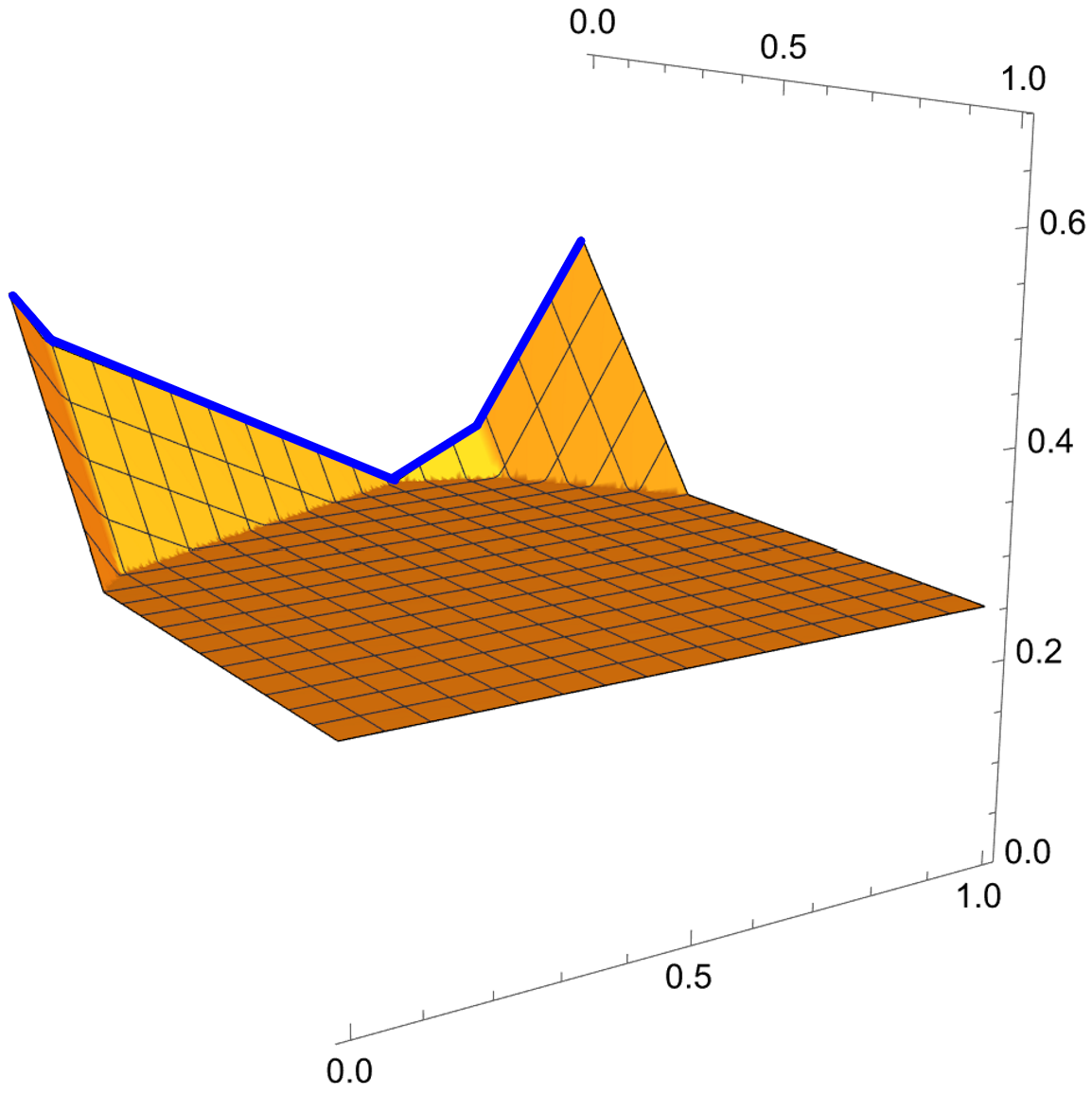}
\caption{An example of a possible function from $f\in \mathcal{F}_{cont}$ (left) together with an illustration of its truncated extension $\hat{f}_0(\x, \y)$ (right) as constructed in \eqref{eq:funcExt}.}
  \label{fig:single_function_extension}
\end{figure}

    
    Given a collection $\tF = (f_{\bar{\x}})_{\bar{\x}}$ with one function $f_{\bar{\x}} \in \xoverline{\mathcal{F}}_{cont}$ for each 0/1 point $\bar{\x}$, we combine them into the convex function
    \begin{align}
    \psi_{\texttt{F}}(\x, \y) := \max_{\bar{\x}\in \{0,1\}^n} \hat{f}_{\bar{\x}}(\x, \y),  \label{eq:psi}
    \end{align}
    where we abuse notation slightly and write the convex extension $(\widehat{f_{\bar \x}})_{\bar\x}$ as simply $\hat{f}_{\bar{\x}}$ to simplify the notation.
    As mentioned above, a crucial property of these functions is that their behavior between the fibers is determined by the behavior on the closest fiber. Intuitively, the slope $\pm 3 MR$ guarantees that as the $\x$ argument moves away from the base fiber $\bar{\x}$ of each extended function $\hat{f}_{\bar{\x}}$, $\hat{f}_{\bar{\x}}$ decreases rapidly enough so that the maximum in \eqref{eq:psi} is always achieved by the extended function at the closest fiber to $\x$. Figure \ref{fig:hardFuncMI} illustrates this, where one can see that both functions placed on the fibers get fully truncated in between the fibers. To make this precise, let $r(\x): [0,1]^n \rightarrow \{0,1\}^n$ map any $\x$ in the box to its closest 0/1 point in $\ell_\infty$-norm, that is $r(\x) := \argmin_{\x'} \{\|\x - \x'\|_\infty: \x' \in \{0,1\}^n\}$.
    
    \begin{lemma}\label{lemma:psi}
    For every collection $\tF = (f_{\bar{\x}})_{\bar{\x}}$, for every point $(\x,\y) \in [0,1]^n \times [-R, R]^d$ we have 
        \begin{align*}
            \psi_{\texttt{F}}(\x,\y) = \hat{f}_{r(\x)}(\x,\y), ~~~~\textrm{ and }~~~~ \partial \psi_{\texttt{F}}(\x,\y) = \partial \hat{f}_{r(\x)}(\x,\y)
        \end{align*}
    \end{lemma}

\begin{proof}
Define $B_{\Bar{\x}}:=\left\{\x' \in \R^n:\|\x' - \Bar{\x}\|_\infty\leq  \frac{1}{3}\right\}$ for every $\bar\x \in \{0,1\}^n$. Consider an arbitrary $(\x,\y) \in [0,1]^n \times [-R, R]^d$.
\medskip

\medskip

\noindent\underline{Case 1: $\x \not\in B_{\bar\x}$ for any $\bar\x \in \{0,1\}^n$.} This implies that for every $\bar\x \in \{0,1\}^n$, we have 

\begin{align*}
    f_{\bar{\x}}(\y) + \langle \mathbf{M}_{\bar{\x}}, \x - \bar{\x}\rangle &= f_{\bar{\x}}(\y)+\sum_{j \in [n]} 3MR \cdot sgn(\bar{\x}_j-0.5) \cdot (\x_j - \bar{\x}_j) \\ &\leq f_{\bar{\x}}(\y)-3MR\cdot \max_{j \in [n]} |\x_j - \bar{\x}_j| < f_{\bar{\x}}(\y) -MR \leq \OPT,
\end{align*}

Thus, $\hat f_{\bar\x}(\x,\y) = \OPT$  for all $\bar\x \in \{0,1\}^n$. As a result, $\psi_F(\x,\y) = \OPT = \hat f_{r(\x)}(\x,\y)$. 

Moreover, since $f_{\bar{\x}}(\y) + \langle \mathbf{M}_{\bar{\x}}, \x - \bar{\x}\rangle<\OPT$ for all $\bar\x \in \{0,1\}^n$, there exists a neighborhood of $(\x,\y)$ such that for any point $(\x',\y')$ in the neighborhood, it holds that $f_{\bar{\x}}(\y') + \langle \mathbf{M}_{\bar{\x}}, \x' - \bar{\x}\rangle< \OPT$, and  $\hat f_{\bar\x}(\x',\y') = \OPT$ for all  $\bar\x \in \{0,1\}^n$. As a result, $\partial \hat f_{\bar\x}(\x,\y)=\{\0\}$ for all  $\bar\x \in \{0,1\}^n$, and $\partial\psi_F(\x,\y) = \{\0\}.$

\medskip

\noindent\underline{Case 2: $\x \in B_{\bar\x}$ for some $\bar\x \in \{0,1\}^n$.} In this case, $r(\x) = \bar\x$. This is because for any $\tilde\x\in \{0,1\}^n\backslash\{\bar\x\}$, we have that 
\begin{align}\label{eq:neightbor-x}
    \|\x - \tilde\x\|_\infty\geq \frac{2}{3}>\frac{1}{3}\geq\|\x - \Bar{\x}\|_\infty.
\end{align}
 It is also true that $\hat{f}_{\Bar{\x}}(\x,\y)\geq\OPT= \hat{f}_{\tilde\x}(\x,\y)$ for any $\tilde\x\neq \Bar{\x}$, which holds due to the result from Case 1. 

Moreover, the arguments from Case 1 and (\ref{eq:neightbor-x}) imply that there exists a neighborhood of $(\x,\y)$ such that for any point $(\x',\y')$ in the neighborhood, it holds that $r(\x') = r(\x) = \Bar{\x}$, and $\hat{f}_{\Bar{\x}}(\x',\y')\geq\OPT= \hat{f}_{\tilde\x}(\x',\y')$ for all $\tilde\x\neq \Bar{\x}$. As a result, $\psi_F(\x',\y') = \hat f_{\Bar\x}(\x',\y') = \hat f_{r(\x')}(\x',\y') = \hat f_{r(\x)}(\x',\y')$ and $\partial\psi_F(x,y)= \partial \hat{f}_{\Bar{\x}}(\x,\y) = \partial\hat{f}_{r(\x)}(\x,\y).$
\end{proof}


\paragraph{Construction of the mixed-integer adversary \textbf{Adv-MI}.} We finally describe \textbf{Adv-MI} in Procedure~\ref{proc:advMI}. Its main property is captured in the following invariant.

    \vspace{4pt}
    \begin{mdframed}
    \vspace{-7pt}
    \begin{proc} \label{proc:advMI}
     \normalfont
     \textbf{Adv-MI} 

    \vspace{4pt}
    \noindent Instantiate a copy of \textbf{Adv-Cont+} on each fiber $\x \in \{0,1\}^n$, and let $S(\x)$ denote the set of surviving functions maintained in every round by this copy, initialized to $\overline{\mathcal{F}}_{cont}$.  
    
    \vspace{6pt}
    \noindent For each round $t = 1,2\ldots$\,:

    \vspace{-8pt}
    \begin{enumerate}[leftmargin=18pt]
        \item  \textbf{Adv-MI} receives the query $(\x_t,\y_t)$ from the algorithm. Send $\y_t$ to the adversary \textbf{Adv-Cont+} associated with the closest fiber $r(\x_t)$, which then returns a value $v$ and subgradient $\g$, and updates its maintained set of surviving functions $S(r(\x_t))$ of its fiber $r(\x_t)$. 
    
        \item \textbf{Adv-MI} returns as its response to the query $(\x_t,\y_t)$ the value 
        \begin{align*}
            \tilde{v}_t = \max\Big\{v +  \langle \mathbf{M}_{r(\x_t)}, \x_t - r(\x_t)\rangle~,~\OPT    \Big\},
        \end{align*}
        and as subgradient returns either $\tilde{\g}_t = (\mathbf{M}_{\bar{\x}}, \g)$ or $\tilde{\g}_t = \zero$ depending whether  $\tilde{v}_t > \OPT$ or not (i.e., whether $\hat{f}_{r(\x)}$ was truncated at $(\x_t, \y_t)$ or not), respectively. 
    \end{enumerate} 
        
\end{proc}
\end{mdframed}
\medskip

\begin{invariant} \label{inv:advMI}
    There exists a first order chart $\cG$ (derived from $\cG_0$) such that, for any algorithm, the sets $S(\x)$, $\x \in \{0,1\}^n$ maintained by \textbf{Adv-MI} satisfy the following property.

    In every round, for every collection $\tF = (f_{\x})_{\x\in \{0,1\}^n}$ of current surviving functions $f_{\x} \in S(\x)$ for $\x \in \{0,1\}^n$, 
    the function $\psi_{\tF}$ is consistent with the response returned by \textbf{Adv-MI} under the full-information first-order oracle $\cO(\cG)$, i.e., $\psi_{\tF}(\x_t,\y_t) = \tilde{v}_t$ and $\tilde{\g}_t \in \partial \psi_{\tF}(\x_t,\y_t)$. 
\end{invariant}
%
\medskip

Notice that Invariant \ref{inv:advMI} is indeed maintained after each response in Step 2 of Procedure~\ref{proc:advMI}: For every collection $\tF = (f_{\bar{\x}})_{\bar{\x}}$ of still surviving functions $f_{\bar{\x}} \in S(\bar{\x})$, by the consistency guarantee of \textbf{Adv-Cont+} (Item 1 of Lemma \ref{lemma:advContPlus}) the function $f_{r(\x_t)}$ selected for the fiber $r(\x_t)$ has value $v$ and subgradient $\g$ at $\y_t$; 
thus, Lemma~\ref{lemma:psi} combined with \eqref{eq:funcExt} implies that the function $\psi_{\tF}$ has value $$\psi_{\tF}(\x_t,\y_t) =   \hat{f}_{r(\x)}(\x_t,\y_t) = \max\Big\{f_{r(\x_t)}(\y_t)  + \langle \mathbf{M}_{r(\x_t)}, \x_t - r(\x_t)\rangle~,~\OPT \Big\} = \tilde{v}_t,$$ and similarly from \eqref{eq:subgradFHat} we see that $\tilde{\g}$ is a subgradient in $\partial \psi_{\tF}(\x_t,\y_t) = \partial \hat{f}_{r(\x_t)}(\x_t,\y_t)$ 
, as desired.

\medskip
We now prove that \textbf{Adv-MI} is $\eps$-hard for $2^n \ell - 1$ rounds; using Lemma \ref{lemma:infoAdv}, this implies Theorem \ref{thm: transfer theorem pure optimization} for the case of full-information first-order oracle. Suppose the optimization algorithm runs for fewer than $2^n \cdot \ell$ iterations. Then there is a fiber $\x^* \in \{0,1\}$ where \textbf{Adv-MI} sent at most $\ell-1$ queries to the adversary \textbf{Adv-Cont+} of the fiber $\x^*$. 
Thus, by the guarantee of the latter (Item 2 of Lemma \ref{lemma:advContPlus}), the surviving set $S(\x^*)$ has some finite collection of functions $f^1_{\x^*}, ..., f^k_{\x^*}$ with no common $\eps$-approximate solution. 
Consider the collections $\tF^1, ..., \tF^k$ of surviving functions that have $f^1_{\x^*}, ..., f^k_{\x^*}$, respectively, for the fiber $\x^*$ and any function $f_{\bar{\x}} \in S(\bar{\x})$ with optimal value $> \OPT +\eps$ for each of the other fibers $\bar{\x} \neq \x^*$, which exist on each of the other fibers by Item 3 of Lemma \ref{lemma:advContPlus}. By Invariant \ref{inv:advMI}, all functions $\psi_{\tF^1}, ..., \psi_{\tF^k}$ are compatible with the responses returned by \textbf{Adv-MI}. The desired $\eps$-hardness of \textbf{Adv-MI} then follows from the following claim, which then concludes the proof. 

\begin{claim}\label{claim: psiF1 vs F2 have disjoint eps sol}
    The functions $\psi_{\tF^1}, ..., \psi_{\tF^k}$ share no common $\eps$-approximate solution. 
\end{claim}

\begin{proof}
    From the construction above, we have that $\tF^\dag:=\tF^1\backslash \{f^1_{\x^*}\} =\tF^2\backslash \{f^2_{\x^*}\}...= \tF^k\backslash \{f^k_{\x^*}\}$. 
    Due to (\ref{eq:psi}) and the definitions of $\tF^1, ..., \tF^k$, for any fiber $\bar{\x}\neq \x^*$ and $f_{\bar{\x}}\in \tF^\dag$, it follows that $\psi_{\tF^1}(\bar{\x},\y) =...= \psi_{\tF^k}(\bar{\x},\y) = f_{\bar{\x}}(\y)> \OPT +\eps$. Thus, the $\eps$-approximate solutions for the functions $\psi_{\tF^1}, ..., \psi_{\tF^k}$ only exist within the fiber $\x^*$. Given that the $\eps$-approximate solutions of $f^1_{\x^*}, ..., f^k_{\x^*}$ are disjoint, and considering that $\psi_{\tF^j}(\x^*,\y) = f^j_{\x^*}(\y)$ for all $j\in [k]$, we can conclude our proof.
\end{proof}


\subsection{Proof of Theorem \ref{thm: transfer theorem pure optimization} for general oracles}

We now prove Theorem \ref{thm: transfer theorem pure optimization} in full generality. We will do this using the exact same family of difficult functions $\psi_{\texttt{F}}$ from \eqref{eq:psi}, and also with the same idea of constructing a mixed-integer adversary that produces its answers by making queries to an adversary for the continuous problems on the fibers. Since the mixed-integer adversary will need to answer queries made in the full $\R^n \times \R^d$ space by making queries in the continuous space $\R^d$ on each fiber, we will require that the set of permissible queries that can be made to the continuous adversary is, in some sense, as rich as the queries allowed in the full space. For example, if one allows full-information queries to be made in $\R^n \times \R^d$, but only binary queries to be made in $\R^d$ to the continuous adversary, one would struggle to determine how the mixed-integer adversary should answer those full-information queries by making only binary queries to the adversaries for the continuous subproblems. 
Specifically, for a query at $(\x, \y)\in \R^n \times \R^d$, knowing how $\x$ affects the function values and subgradients of $\psi_\tF$, the mixed-integer adversary needs to be able to determine what response to give by making a suitably chosen query about $f_{r(\x)}$ to the continuous adversary. 
We formalize this requirement of having the same richness of queries for the continuous subproblems as for the full $\R^n \times \R^d$ space with the concept of {\em hereditary queries}. 

\paragraph{Hereditary queries.} For simplicity, we define the notion of hereditary queries for unconstrained problems (i.e., only for value/subgradient queries), but we remark that the same idea can be applied to separation queries as well. 

\begin{definition}\label{def:hereditary}
Let $\{\cH_{n,d}^{\val}\}_{n,d \in \N}$ and $\{\cH_{n,d}^{\sub}\}_{n,d \in \N}$ be classes of permissible function value and subgradient queries, respectively, with response sets (codomains) $H_{n,d}^{\val}$ and $H_{n,d}^{\sub}$. $\{\cH_{n,d}^{\val}\}_{n,d \in \N}$ and $\{\cH_{n,d}^{\sub}\}_{n,d \in \N}$ are said to be \emph{hereditary} if the following holds for all $n, d \in \N$ and functions $\mathcal{M}: \{0,1\}^n \rightarrow \R^n$. 
For any  $\x \in \{0,1\}^n$, $\delta \in \R$, $h^{\val} \in \cH_{n,d}^{\val}$, and $h^{\sub} \in \cH_{n,d}^{\sub}$, 
there exists $h_*^\val \in \cH_{0,d}^{\val}$, 
$h_*^{\sub} \in \cH_{0,d}^{\sub}$ and functions $B^\val: H_{0,d}^{\val}\to H_{n,d}^{\val}$, $B^\sub: H_{0,d}^{\sub}\to H_{n,d}^{\sub}$ 
such that  
\begin{align}
    B^\val(h_*^\val(v)) &= h_\val(v+\delta)\qquad \forall v \in \R, \label{eq: hereditary function value condition}\\
    B^\sub(h_*^\sub(v,\g)) &= h_\sub(\mathcal{M}(\x), \g)\hspace{2mm} \forall \g \in \R^d \label{eq: hereditary subgradient condition}
\end{align}

\end{definition}


Intuitively, a class of queries being hereditary has the consequence that if for a point $(\x, \y) \in \R^n \times \R^d$, one knows exactly the $\x$ component $\cM(\x)$ of the subgradient, then one can simulate a query in the $\R^n \times \R^d$ space by only making a query on the $\R^d$ space, and similarly that there are queries rich enough to consider shifted function values $v+\delta$, where the interpretation is that $\delta$ is the effect $\x$ has on the overall function value -- see~\eqref{eq:funcExt}. 


\begin{example}
We show that natural premissible queries, as from Definition \ref{def:oracle-example}, are hereditary. Let $\cM(\x), \delta$ be as in Definition \ref{def:hereditary}.

\begin{enumerate}
    \item (Full-information first-order oracle) If $\cH_{n,d}^{\val}$ and $\cH_{n,d}^{\sub}$ are simply the identity functions, then we can take $B^\val$ to be $B^\val(v) = v+\delta$ and take $B^\sub$ to be the ``lifting/rotation" map $B^\sub(\g) = (\mathcal{M}(\x), \g)$, noting that $h^\sub, ~h^\val, ~h_*^\sub$, and $h_*^\val$ are all the identity functions. 
    
    \item (General binary oracle) For the general binary oracle based on a first-order chart $\cG$, $B^\val$ and $B^\sub$ can be taken to be the identity map from $\{0,1\}$ to $\{0,1\}$, and one can take $h_*^\val(v) = h(v + \delta), ~h_*^\sub =  h(\mathcal{M}(\x), \g),$ which are permissible queries since all binary queries are permissible. 
    
    \item 
    (Shifted bit oracle) If $\cH_{n,d}^{\val}$ and $\cH_{n,d}^{\sub}$ are from a shifted bit oracle $\cH^{bit^*}$, then $B^\val$ can be taken to be the identity map from $\{0,1\}$ to $\{0,1\}$. For a query on the function value, if $h^\val$ reports some bit of $v+\delta$, then the appropriate hereditary query is exactly the query $h_*^\val$ such that  $h_*^\val(v) = h(v+\delta)$, i.e. using the shift $u = \delta$ in the notation of Definition \ref{def:oracle-example}. A subgradient bit query $h^\sub(\mathcal{M}(\x), \g)$ returns a bit of either $\mathcal{M}(\x)$ or $\g$, so there are two cases. 
    \begin{itemize}
        \item[i)] $h^\sub$ returns the $j^{th}$ bit of the $k^{th}$ entry of $\mathcal{M}(\x)$. Set $B^\sub(\cdot)$ to return exactly that bit of $\mathcal{M}(\x)$, no matter the input to $B^\sub$, so $h_*^\sub$ may be chosen arbitrarily. 
        \item[ii)] $h^\sub$ returns the $j^{th}$ bit of the $k^{th}$ entry of $\g$. Set $B^\sub$ to be the identity map from $\{0,1\}$ to $\{0, 1\}$, and set $h_*^\sub$ to return the desired bit of $\g$.
    \end{itemize}

    \item (Inner product threshold queries) If $\cH_{n,d}^{\val}$ and $\cH_{n,d}^{\sub}$ consist of the inner product threshold queries, take both $B^\val$ and $B^\sub$ to be the identity maps. For function value queries $h^\val_{u, c}(v) = sgn (u\cdot  (v+\delta) -c)$, use 
    $$h_*^\val(v) := h_{u, c-u\delta}^\val(v) = sgn (uv - (c-u\delta)),$$ 
    since then $h_*^\val(v) = sgn (uv - (c-u\delta)) = sgn (u\cdot  (v+\delta) -c) = h^\val_{u, c}(v)$ as desired. \\ 
    For subgradient queries $h^\sub_\u(\mathcal{M}(x), \g) = sgn (\langle \u, \mathcal{M}(x), \g \rangle-c),$ with $\u \in \R^{n+d}$, denote by $\u_n$ the vector of the first $n$ entries of $\u$, and by $\u_d$ the vector of the last $d$ entries of $\u$. One may use
    $$h_*^\sub(\g) := h_{\u_d, c-\langle \u_n, \mathcal{M}(\x)\rangle}^\sub (\g)= sgn \Big(\langle \u_d, \g \rangle - (c-\langle \u_n, \mathcal{M}(\x) \rangle)\Big),$$
    since then we similarly have 
    $$h_*^\sub(\g) = sgn \Big(\langle \u_d, \g \rangle - (c-\langle \u_n, \mathcal{M}(\x) \rangle)\Big) = sgn (\langle \u, (\mathcal{M}(x), \g) \rangle-c) = h^\sub_\u(\mathcal{M}(x), \g)$$ 
    as desired. For these hereditary queries, note that $h_{u, c-u\delta}^\val$ and $h_{\u_d, c-\langle \u_n, \mathcal{M}(\x)\rangle}^\sub$ are indeed in $\cH^\val_{0, d}$ and $\cH^\sub_{0, d}$, respectively, since $u\in \R$, $\u_n \in \R^n$, and $\u_d \in \R^d$. 
    
\end{enumerate}
\end{example}

 We remark here that $\cH^{bit}$, without the permitted ``shifts" allowed in $\cH^{bit^*}$ is not hereditary, as it may not satisfy condition~\eqref{eq: hereditary function value condition} for the function values for all $\delta$; however, any lower bounds obtained with $\cH^{bit^*}$ must also hold for $\cH^{bit}$, since the former is a richer class of queries. 


\paragraph{Definition of the adversary.} 
We will define \textbf{Adv-Cont+} analogously as in the full-information case, now receiving queries in $\cH$ according to the oracle setting considered. \textbf{Adv-Cont+} will be $\eps$-hard for $\ell$ rounds answering queries from $\cH$, and after $\ell-1$ rounds it commits to a single surviving function with optimal value $> OPT + \eps$. As queries for general oracles using first-order information consist of a point $\z = (\x, \y) \in \R^n \times \R^d$ and a permissible query $h\in \cH$, let us write $(\x, \y, h)$ for notational simplicity to denote such a query.  

We describe here the behavior of \textbf{Adv-Cont+} in the general oracle case, and such that it satisfies the same invariant of Lemma \ref{lemma:advContPlus} as in the full-information case, i.e. it is $\eps$-hard for $\ell$ rounds and only keeps a single surviving function with optimal value at least $OPT+\eps$ after $\ell$ queries have been made. 

  \vspace{4pt}
    \begin{mdframed}
    \vspace{-7pt}
    \begin{proc} \label{proc:advCont_general}
     \normalfont
     \textbf{Adv-Cont+} 

    \vspace{4pt}
    \noindent Initialize set of survived functions $S_0 = \xoverline{\mathcal{F}}_{cont}$

    \vspace{4pt}
    \noindent For each round $t = 1,2\ldots$\,:

    \vspace{-8pt}
    \begin{enumerate}[leftmargin=18pt]
        \item Receive query point $(\y_t, h_t) \in \R^d \times \cH$  from the optimization algorithm.
        
        \item  \vspace{-3pt}  If $t \le \ell - 1$: Send $(\y_t, h_t)$ to the adversary \textbf{Adv-Cont}, receiving back the answer $\alpha$. Obtain $S_{t}$ by removing from $S_{t-1}$ the functions $f$ that are not consistent with this response under any first-order chart $\cG$, namely $f$ for which $h_t(f(\y_t)) \not = \alpha$ if $h_t\in \cH^\val$, or for which there does not exist a $\g_t \in \partial f(\y_t)$ such that $h_t(f(\y_t),\g_t) = \alpha$ if $h_t\in \cH^\sub$.
        
        Send the response $\alpha$ to the optimization algorithm. 
       
         \item \vspace{-3pt} If $t = \ell$: Since \textbf{Adv-Cont} is $\eps$-hard for $\ell-1$ rounds, there is a finite collection of functions $\{f_1, ..., f_k\} \subset S_{t-1} \cap \mathcal{F}_{cont}$ that do not share an $\eps$-solution. Define their pointwise maxima $f_{\max} = \max\{f_1, ..., f_k\}$ and set $S_{t+k} = \{f_{\max}\}$, for all $k = 0, 1, 2...$. 
        
        Set the value $v_t$ to be $f_{\max}(\y_t)$ and set $\g_t$ to be a subgradient in $\partial f_{\max}(\y_t)$ (consistent with what the first order chart $\cG_0$ gives for $f_1, \ldots, f_k$ at $\y_t$, if $\y_t$ has been queried in an earlier round), and send the response $h_t(v_t)$ or $h_t(\g_t)$ to the optimization algorithm, according to whether $h_t\in \cH^\val$ or $h_t\in \cH^\sub$ respectively.
          
        \item \vspace{-3pt} If $t > \ell$: Let $f_{\max}$ be the only function in $S_{t-1}$.
        If $\y_t$ was queried in an earlier round $k$, answer $(v_k, \g_k)$. Otherwise, 
        as in the step above, set the value $v_t$ to be $f_{\max}(\y_t)$ and set $\g_t$ to be any subgradient in $\partial f_{\max}(\y_t)$, and again send the appropriate response $h_t(v_t)$ or $h_t(\g_t)$ to the optimization algorithm.
\end{enumerate}
\end{proc}
\end{mdframed}
\medskip

Proving that this \textbf{Adv-Cont+} satisfies the invariant from Lemma \ref{lemma:advContPlus} follows exactly the same steps as in the full-information case. Hence, using this \textbf{Adv-Cont+} and the same family of functions $\psi_{\texttt{F}}$ from \eqref{eq:psi}, we will be able to construct \textbf{Adv-MI} for this general oracle case to satisfy a version of Invariant \ref{inv:advMI}, slightly modified for this general case to refine what we mean by functions being consistent with responses given to the more general queries. To achieve this, let \textbf{Adv-MI} operate according to the following procedure.
 \begin{mdframed}
    \vspace{-7pt}
    \begin{proc} \label{proc:advMI_general}
     \normalfont
     \textbf{Adv-MI} 

    \vspace{4pt}
    \noindent Instantiate a copy of \textbf{Adv-Cont+} on each fiber $\x \in \{0,1\}^n$, and let $S(\x)$ denote the set of surviving functions maintained by this copy, initialized to $\overline{\mathcal{F}}_{cont}$.  
    
    \vspace{6pt}
    \noindent Set $\mathbf{M}_{x} := 3MR\cdot sgn(\x-0.5\cdot\ones)$, for any $\x\in\{0,1\}^n$, as in \eqref{eq:funcExt}. For each round $t = 1,2\ldots$\,:

    \vspace{-8pt}
    \begin{enumerate}[leftmargin=18pt]
        \item  \textbf{Adv-MI}  receives the query $(\x_t,\y_t, h_t)$ from the algorithm. Set $\delta := \langle \mathbf{M}_{r(\x_t)}, \x_t - r(\x_t)\rangle$. 
        \item Send the function value threshold query that answers``is $f_{r(\x_t)}(\y_t) \leq OPT + \delta$?" to the adversary \textbf{Adv-Cont+} of the closest fiber $r(\x_t)$, which responds and updates its set of surviving functions $S(r(\x_t))$.
        \vspace{0.5cm}\\
        \textit{If} the answer is yes, \textbf{Adv-MI} responds $h_t(OPT)$ or $h_t(\zero)$ to the original query, according to whether $h_t\in \cH^\val$ or $h_t\in \cH^\sub$, respectively.
        \vspace{0.5cm}\\
        \textit{Else}, determine an appropriate $B$ and hereditary query $h_*$ as from Definition~\ref{def:hereditary} using $r(\x_t)$, $\mathbf{M}_{r(\x_t)}$, $\delta$ and $h_t$, and send the query $(\y_t, h_*)$ to the adversary \textbf{Adv-Cont+} of the closest fiber $r(\x_t)$, which then returns some answer $\alpha$ and updates its set of surviving functions $S(r(\x_t))$. \textbf{Adv-MI} returns $B(\alpha)$ as its response to the original query. 
    \end{enumerate} 
        
\end{proc}
\end{mdframed}
\medskip

\begin{invariant} \label{inv:advMI general}

There exists a first order chart $\cG$ (derived from $\cG_0$) such that, for any algorithm, the sets $S(\x)$, $\x \in \{0,1\}^n$ maintained by \textbf{Adv-MI} satisfy the following property.

     In every round, for every collection $\tF = (f_{\x})_{\x\in \{0,1\}^n}$ of current surviving functions $f_{\x} \in S(\x)$ for $\x \in \{0,1\}^n$, 
    the function $\psi_{\tF}$ is consistent with the response returned by \textbf{Adv-MI} under under the oracle $\cO(\cH, \cG)$, i.e., the response \textbf{Adv-MI} gives is equal to $h_t(\tilde{v}_t, \tilde{\g}_t)$ for $\tilde{v}_t = \psi_{\texttt{F}}(\x_t, \y_t)$ and some $\tilde{\g}_t\in \partial \psi_{\texttt{F}}(\x_t, \y_t)$. 


\end{invariant}

We claim that the Invariant \ref{inv:advMI general} is maintained after each response in Step 2 of Procedure~\ref{proc:advMI_general}. 

If \textbf{Adv-Cont+} answers yes, then $\psi_{F}(\x_t, \y_t) = OPT$ and $\zero \in \partial \psi_{F}(\x_t, \y_t)$ for any collection of surviving functions $\tF = (f_{\bar{\x}})_{\bar{\x}}$, since all the extensions $\hat{f}_{\bar\x}$ as in \eqref{eq:funcExt} that are consistent with this affirmative response in Step 2 must be truncated at $(\x_t, \y_t)$. To see this, note that due to the choice of $\delta$, the query ``is $f_{r(\x_t)}(\y_t) \leq OPT + \delta$?" is equivalent to asking whether $\hat{f}_{r(\x_t)}$ has value $OPT$ at $(\x_t, \y_t)$ (i.e., was truncated), and Lemma \ref{lemma:psi} guarantees that we only need to consider $\hat{f}_{r(\x_t)}$ for $\psi_\tF(\x_t, \y_t)$.

If the answer is no, then Lemma~\ref{lemma:psi} combined with~\eqref{eq:funcExt} implies that for every choice $\tF$ of surviving functions from each fiber, the function $\psi_{\tF}$ has value 
\begin{align}\psi_{\tF}(\x_t,\y_t) =   \hat{f}_{r(\x)}(\x_t,\y_t) = \max\Big\{f_{r(\x_t)}(\y_t)  + \langle \mathbf{M}_{r(\x_t)}, \x_t - r(\x_t)\rangle~,~\OPT \Big\} = f_{r(\x_t)}(\y_t) + \delta, \label{eq: psi hereditary values}
\end{align}
and for its subgradient we have 
\begin{align}
    \g \in \partial \hat{f}_{r(\x_t)} \implies  (\mathbf{M}_{r(\x_t)}, \g) \in \partial \psi_\tF(\x_t, \y_t), \label{eq: psi hereditary subgrads}
\end{align}
by Lemma \ref{lemma:psi} and~\eqref{eq:subgradFHat}. 

Suppose first that $h_t\in \cH^\val$ was a function value query, and denote by $B^\val$ and $h_*^\val$ the transformation and hereditary query that \textbf{Adv-MI} uses, according to definition \ref{def:hereditary}, giving $B^\val(h_*^\val(v)) = h_t(v+\delta)$ for all $v\in \R$. For every collection $\tF = (f_{\bar{\x}})_{\bar{\x}}$ of surviving functions $f_{\bar{\x}} \in S(\bar{\x})$, by the consistency guarantee of \textbf{Adv-Cont+} (Item 1 of Lemma \ref{lemma:advContPlus}), the function $f_{r(\x_t)}$ selected for the fiber $r(\x_t)$ has response $h_*^\val(f_{r(\x_t)}(\y_t)) = \alpha$. Then, from the definition of hereditary queries and \eqref{eq: psi hereditary values}, we have $B^\val(h_*^\val(f_{r(\x_t)}(\y_t))) = h_t(f_{r(\x_t)}(\y_t)+ \delta) = h_t(\psi_{\tF}(\x_t, \y_t))$, and so $\psi_\tF$ is indeed consistent with the response $B^\val(\alpha)$ provided by \textbf{Adv-MI}.

If instead $h \in \cH^\sub$ was a subgradient query, again denote $B^\sub$ and $h_*^\sub$ as the appropriate transformation and hereditary query, with $B^\sub(h_*^\sub(\g)) = h(\mathcal{M}(\x), \g)$ for all $\g \in \R^d$. Again, for every choice $\tF$ of surviving functions on the fibers, the function $f_{r(\x_t)}$ on $r(\x_t)$ has $h_*^\sub(\g) = \alpha$, for some $\g\in \partial f_{r(\x_t)}(\y_t)$. 
Then, from the definition of hereditary queries and \eqref{eq: psi hereditary subgrads}, $B^\sub(h_*^\sub(\g)) = h_t(\mathbf{M}_{r(\x_t)}, \g) = h_t(\g_\psi)$, with $\g_\psi\in \partial \psi_{\tF}(\x_t, \y_t)$. Hence, whether $h_t$ is a function value or subgradient query, all functions $\psi_\tF$ for choices $\tF$ of the surviving functions on the fibers are consistent with the responses given by \textbf{Adv-MI}, for the oracle $\cO(\cG,\cH)$ with permissible queries $\cH$ and the first-order chart $\cG$ from Invariant~\ref{inv:advMI general}.

We now prove that \textbf{Adv-MI} is $\eps$-hard for $2^{n-1} \ell - 1$ rounds, thus proving Theorem \ref{thm: transfer theorem pure optimization} in the general case. 
Since \textbf{Adv-MI} makes at most $2$ queries to \textbf{Adv-Cont+} in every round (Step 2) of Procedure~\ref{proc:advMI_general}, if the optimization algorithm runs for fewer than $2^{n-1} \cdot \ell$ iterations, there is a fiber $\x^* \in \{0,1\}$ where \textbf{Adv-MI} sent at most $\ell-1$ hereditary queries to the adversary \textbf{Adv-Cont+} of the fiber $\x^*$. 
Thus, by the guarantee of the latter (Item 2 of Lemma \ref{lemma:advContPlus}), the surviving set $S(\x^*)$ has some finite collection of functions $f^1_{\x^*}, ..., f^k_{\x^*}$ with no common $\eps$-approximate solution, and the remainder of the proof follows as in the full-information case, by considering $\psi_{\tF^1}, ..., \psi_{\tF^k}$ that have $f^1_{\x^*}, ..., f^k_{\x^*}$ on the fiber $\x^*$, and some functions with optimal value greater than $OPT + \eps$ on all the other fibers.

\section{Proof of Theorem~\ref{thm:binary-LB}}\label{sec:information-mem}

To demonstrate Theorem~\ref{thm:binary-LB}, we need to introduce the idea of {\em information memory} of any query strategy/algorithm.
\medskip

\begin{definition}
A first-order query strategy with {\em information memory} comprises  three functions:

\begin{enumerate}
    \item $\phi_{\query}: \{0,1\}^* \to [-R,R]^n \times [-R,R]^d$
    \item $\phi^{\sepp}_{\update}: \left(\R^n \times \R^d\right) \times\{0,1\}^* \to \{0,1\}^*$
    \item $\phi^{\val}_{\update}:\R  \times\{0,1\}^* \to \{0,1\}^*$
    \item $\phi^{\sub}_{\update}:\left(\R^n \times \R^d\right) \times\{0,1\}^* \to \{0,1\}^*$,
\end{enumerate}
where $\{0,1\}^*$ denotes the set of all binary strings (finite sequences over $\{0,1\}$), including the empty string.

Given access to a first-order chart $\mathcal{G}$, the query strategy maintains an {\em information memory} $r_k$ at every iteration $k \geq 0$, which is a finite length binary string in $\{0,1\}^*$, with $r_0$ initialized as the empty string. At every iteration $k = 1,2,\ldots$, the query strategy computes $\z_k := \phi_{query}(r_{k-1})$ and updates its memory using either $r_k = \phi^{\sepp}_{\update}\left(\g^{\sepp}_{\z_k}(\widehat f, \widehat C), r_{k-1}\right)$, $r_k = \phi^{\val}_{\update}\left(\g^{\val}_{\z_k}(\widehat f, \widehat C),r_{k-1}\right)$ or $r_k = \phi^{\sub}_{\update}\left(\g^{\sub}_{\z_k}(\widehat f, \widehat C),r_{k-1}\right)$, where $(\widehat f, \widehat C)$ is the unknown true instance. After finitely many iterations, the query strategy does a final computation based on its information memory and reports an $\epsilon$-approximate solution, i.e., there is a final function $\phi_{\textup{fin}}:\{0,1\}^* \to \Z^n\times \R^d$.

The {\em information memory complexity} of an algorithm for an instance is the maximum length of its information memory $r_k$ over all iterations $k$ during the processing of this instance.

\end{definition}

The following proposition allows us to relate the information memory complexity of first-order algorithms with information complexity under access to a general binary oracle using first-order information. 

\begin{prop}\label{prop:reduction}
Let $\mathcal{G}$ be a first-order chart. For any first-order query strategy $\mathcal{A}$ with information memory that uses $\mathcal{G}$, there exists a query strategy $\mathcal{A}'$ using the general binary oracle based on $\mathcal{G}$, such that for any instance $(f,C)$, if $\mathcal{A}$ stops after $T$ iterations with information memory complexity $Q$, $\mathcal{A}'$ stops after making at most $Q\cdot T$ oracle queries.

Conversely, for any query strategy $\mathcal{A}'$ using the general binary oracle based on $\mathcal{G}$, there exists a first-order query strategy $\mathcal{A}$ with information memory such that for any instance $(f,C)$, if $\mathcal{A}'$ stops after $T$ iterations, $\mathcal{A}$ stops after making at most $T$ iterations with information memory complexity at most $T$.
\end{prop} 

\begin{proof}
Let $\mathcal{A}$ be a first-order query strategy with information memory. We can simulate $\mathcal{A}$ by the query strategy whose queries are precisely the bits of the information memory state $r_k$ at each iteration $k$ of $\mathcal{A}$. 
More formally, the query is $\z= \phi_{\query}(r_{k-1})$ and $h(\cdot) = (\phi^{\sepp}_{\textrm{update}}(\cdot, r_{k-1}))_i$, $h(\cdot) = (\phi^{\val}_{\textrm{update}}(\cdot, r_{k-1}))_i$, or $h(\cdot) = (\phi^{\sub}_{\textrm{update}}(\cdot, r_{k-1}))_i$, depending on which type of query was made, where $i$ indexes different bits of the corresponding binary string.

Conversely, given a query strategy $\mathcal{A}'$ based on the general binary oracle, we can simulate it with a first-order query strategy with information memory where in each iteration, we simply append the new bit queried by $\mathcal{A}'$ to the current state of the memory.
\end{proof}

We need the following result derived from~Marsden et al. \cite{marsden2022efficient} on information memory complexity. 
\medskip

\begin{theorem}\label{thm:mem-constrained} \textup{\cite[Theorem 1]{marsden2022efficient}} For every $\delta \in [0,1/4]$, there is a class of instances $\mathcal{I}\subseteq \mathcal{I}_{n,d,R,\rho,M}$, where $n=0$, and a first-order chart $\mathcal{G}$ such that any first-order query strategy with information memory must have either $d^{1.25 - \delta}$ information memory complexity (in the worst case) or make at least $\tilde\Omega(d^{1+\frac{4}{3}\delta})$ iterations (in the worst case). 
\end{theorem}

\begin{proof}[Proof of Theorem~\ref{thm:binary-LB}] In the case when $n = 0$, we can set $\delta = \frac{3}{28}$ in Theorem~\ref{thm:mem-constrained} to obtain that any first-order query strategy uses either $d^{8/7}$ information memory or makes at least $\tilde\Omega(d^{8/7})$ iterations. Using the second part of Proposition~\ref{prop:reduction}, we obtain the lower bound of $\tilde\Omega(d^{8/7})$ on the number of queries made by any query strategy using the general binary oracle based on $\mathcal{G}$.

Applying Theorem \ref{thm: transfer theorem pure optimization} enables us to extend the bound to the mixed-integer scenario ($n>0$). Further, by integrating this with Corollary \ref{cor:standard-LB}, we can obtain the desired bound.
\end{proof}


\section{Proof of Theorems \ref{thm:binary-UB-mixed} and \ref{thm:binary-UB-cont}}

We will use $B_\infty(\p,\delta)$ to denote the $\ell_\infty$ ball of radius $\delta$ centered at $\p \in \R^n \times \R^d$, i.e., $B_\infty(\p,\delta) = \{\z \in \R^n \times \R^d: \|\z - \p\|_\infty \leq \delta\}$. Recall that we consider the subclass of instances $(f,C) \in \I_{n,d,R,\rho,M}$ such that the fiber containing the optimal solution also contains a point that is {\em $\rho$-deep in $C$}, that is: if $(\x^*,\y^*) \in \Z^n \times \R^d$ is an optimal solution for this instance, then there is a point  $(\x^*, \bar{\y})$ such that the full-dimensional ball $B_\infty((\x^*, \bar{\y}), \rho)$ is contained in $C$. We use $\I_{n,d,R,\rho,M}^{deep}$ to denote this subclass of instances. We will use $C_{-\rho}:=\{\z\in C:B_\infty(\z,\rho)\subseteq C\}$ to denote the set of all $\rho$-deep points in $C$.

Our strategy for proving Theorems \ref{thm:binary-UB-mixed} and \ref{thm:binary-UB-cont} is to: 1) solve the the problems using approximate subgradients/separating hyperplanes; 2) use bit queries/inner product sign queries to construct such approximations.

For the first item, we use an algorithm designed by Oertel~\cite{oertel2014integer} (see also~\cite{basu2017centerpoints}) 
based on the concept of a \emph{centerpoint}: this is a point in the convex set where every halfspace supported on it  cuts off a significant (mixed-integer) volume of the set. 
The algorithm maintains an outer relaxation $P$ of the feasible region $C$ in every iteration, and repeatedly applies separation or subgradient-based cuts through the centerpoint of $P$. The assumption that the feasible region contains a ball (in the optimal fiber) establishes a volume lower bound that essentially limits the number of iterations of the algorithm. While the original algorithm in~\cite{oertel2014integer,basu2017centerpoints} uses exact separation/subgradient oracles, we show, not surprisingly, that approximate ones suffice. To prove Theorem \ref{thm:binary-UB-cont}, we employ a similar approach. However, due to the continuous nature of the setting, we can obtain a better upper bound compared to Theorem \ref{thm:binary-UB-mixed} by applying a stronger bound on the centerpoints from Gr\"unbaum \cite{Gruenbaum1960}.

The next item is to construct approximate separation/subgradient oracles by making only a limited number of binary queries on the separating hyperplanes and/or subgradients. In case of bit queries $\mathcal{H}^{\textup{bit}}$ this can be easily done by querying enough bits of the latter. The case of inner product sign queries $\mathcal{H}^{\textup{dir}}$, where we can pick a direction $\a$ and ask ``Is $\ip{\a}{\g} \geq 0$?'' for the subgradient or separating hyperplane $\g$, is more interesting. It boils down to approximating the vector $\g$ (subgradient/separating hyperplane) using few such queries.\footnote{This is related to (actively) learning the linear classifier whose normal is given by $\g$~\cite{activeLearnHalf}. These methods can perhaps be adapted to our setting, but we present a different and self-contained statement and proof. See the discussion at the end of Section~\ref{sec: proof of bit approximate}.}
\bigskip

To formalize the first item, we begin by defining three approximate oracles as follows. 

\medskip

	\begin{definition} \label{def:approxO}
		We have the following:
		\begin{itemize}
			\item An \emph{$\eps$-approximate separation oracle} $\hat{\g}^{\sepp}$ is such that 
$\hat{\g}^{\sepp}_{\bar{\z}}(f,C) = \0$ iff $\bar\z$ belongs to $C$, and otherwise the cut $\ip{\hat{\g}^{\sepp}_{\bar{\z}}(f,C)}{\z} \le \ip{\hat{\g}^{\sepp}_{\bar{\z}}(f,C)}{\bar{\z}}$ is valid for all {\em $\eps$-deep points} $\z \in C_{-\eps}$.
		
			\item An \emph{$\eps$-approximate value cut oracle} $\hat{\g}^{\sub}$ is such that 
   for every $\z$ such that $\ip{\hat{\g}^{\sub}_{\bar{\z}}(f, C)}{\z} \geq \ip{\hat{\g}^{\sub}_{\bar{\z}}(f, C)}{\bar{\z}}$, we have $f(\z) \geq f(\bar{\z}) - \eps$. 
			
			\item An \emph{$\eps$-approximate value comparison oracle} is such that for every function $f : [-R,R]^{n+d} \rightarrow [-U,U]$ and every pair of points $\z,\z'$ we obtain the answer to the query ``Is $f(\z) \le f(\z') + \eps$?''.
		\end{itemize}
	\end{definition}

 \medskip

\noindent Then the first item can be formalized as the following. 
\renewcommand{\e}{\varepsilon}

 \medskip

	\begin{theorem} \label{thm:UBOpt}
There exists an algorithm that, for any $M,R >0$, $0 < \e \leq MR$ and $\rho > 0$, can report an $\e$-approximate solution for every instance in $\I_{n,d,R,\rho,M}^{deep}$, using at most $$O\bigg(2^n (n+d) d \log \bigg(\frac{MR}{\min\{\rho,1\}\e} \bigg) \bigg)$$ oracle calls, given access to any $\rho'$-approximate separation oracle, $\e'$-approximate value cut oracle, and $\e'$-approximate value comparison oracle with $\rho' = \frac{\e' \rho}{4 M R}$ and $\e' = \frac{\e}{6}$.

For the continuous setting with $n=0$, the bound can be improved to $$O\bigg( d \log \bigg(\frac{MR}{\min\{\rho,1\}\e} \bigg) \bigg).$$

	\end{theorem}

		
  

 \medskip
 
\noindent We postpone the proof of Theorem \ref{thm:UBOpt} to Section \ref{sec: proof of first item}.

\bigskip

The next lemma shows that one can implement the approximate oracles from Definition~\ref{def:approxO} using bit queries and inner product sign queries.
\medskip
\medskip

\begin{lemma}\label{lemma:bitQuery-and-dirDer}
		Consider a first-order chart $\mathcal{G}$. Let $f : \R^{n+d} \rightarrow \R$ be a convex $M$-Lipschitz function taking values in $[-U,U]$, and $C \subseteq [-R,R]^{n+d}$ a convex set. 
		
		Then for every pair of points $\bar{\z}, \bar{\z}' \in [-R,R]^{n+d}$, we can obtain an $\e$-approximate separation oracle vector $\hat{\g}^{\sepp}_{\bar{\z}}(f,C)$, an $\e$-approximate value cut vector $\hat{\g}^{\sub}_{\bar{\z}}(f, C)$, and an $\e$-approximate value comparison between $\bar{\z}$ and $\bar{\z}'$ using either a sequence of bit queries  from $\mathcal{H}^{\textup{bit}}$, or a sequence of inner product sign queries from $\mathcal{H}^{\textup{dir}}$, on the separating hyperplane $\g^{\sepp}_{\bar{\z}}(f,C)$, the subgradient $\g^{\sub}_{\bar{\z}}(f, C)$ and  the function value $\g^{\val}_{\bar{\z}}(f, C)$. The number of required queries to implement the approximate oracles is $O\left((n+d) \log \frac{(n+d) R}{\e}\right)$, $O\left((n+d) \log \frac{(n+d) M R}{\e}\right)$ and $O\left(\log \frac{U}{\e}\right)$ respectively, for both $\mathcal{H}^{\textup{bit}}$ and $\mathcal{H}^{\textup{dir}}$.
	\end{lemma}

 \noindent The proof of Lemma \ref{lemma:bitQuery-and-dirDer} is deferred to Section \ref{sec: proof of bit approximate}.
\bigskip

\begin{proof}Theorems \ref{thm:binary-UB-mixed} and \ref{thm:binary-UB-cont} follow from Theorem \ref{thm:UBOpt} and Lemma \ref{lemma:bitQuery-and-dirDer}.
 \end{proof}

\subsection{Proof of Theorem \ref{thm:UBOpt}}\label{sec: proof of first item}
We first describe the centerpoint algorithm for  convex optimization due to Oertel~\cite{oertel2014integer} (see also~\cite{basu2017centerpoints}).  Let the \emph{mixed-integer volume} of a (Borel) set $U \in \R^{n+d}$ be $\mu(U) := \sum_{\x \in \Z^n} \vol_d(U \cap (\{\x\} \times \R^d))$, where $\vol_d$ is the $d$-dimensional Lebesgue measure. The following notion is the main element of the algorithm.
	\medskip
 
	\begin{theorem}[Mixed-integer centerpoint~\cite{oertel2014integer,basu2017centerpoints}] \label{lemma:centerpoint}
		For any compact convex set $C \subseteq \R^{n+d}$, there is a point $\z \in C \cap (\Z^n \times \R^d)$ (called a \emph{mixed-integer centerpoint}) such that for every halfspace $H$ with $\z$ on its boundary, we have $\mu(C \cap H) \geq \frac{1}{2^n (d+1)}\, \mu(C)$. 
	\end{theorem}
	
	Algorithm \ref{alg:centerpoints} below is the centerpoint-based algorithm for solving mixed-integer convex optimization problems from \cite{oertel2014integer}, restated in terms of approximate separation ($\hat{\g}^{\sepp}_{\bar{\z}}(f, C)$), value cut ($\hat{\g}^{\sub}_{\bar{\z}}(f, C)$) and value comparison oracles. 	
	
	\begin{algorithm}[!h]
	\caption{Centerpoints}
	\label{alg:centerpoints}
	 \begin{enumerate}[leftmargin=5pt]
	 	\item  Initialize the version set $P_0 := [-R,R]^{n+d}$ and the collection of feasible points $F = \emptyset$. For iterations $t=0,\ldots,T-1$:
	 
	 \begin{enumerate}
	 		\item Let $\z_t \in P_t \cap (\Z^n \times \R^d)$ be a mixed-integer centerpoint of $P_t$ given by Lemma \ref{lemma:centerpoint}.
	 		
	 		\item If the $\rho'$-approximate separation oracle says that $\z_t$ is infeasible for $C$, add the cut $\ip{\hat{\g}^{\sepp}_{\z_t}(f, C)}{\z} \le \ip{\hat{\g}^{\sepp}_{\z_t}(f, C)}{\z_t}$ to $P_t$, namely set $P_{t+1} = P_t \cap \{\z : \ip{\hat{\g}^{\sepp}_{\z_t}(f, C)}{\z} \le \ip{\hat{\g}^{\sepp}_{\z_t}(f, C)}{\z_t}\}$
	 		
	 		\item Else, add $\z_t$ to the set of feasible solutions $F$, and add the cut from the $\e'$-approximate value cut oracle, namely set $P_{t+1} = P_t \cap \{\z : \ip{\hat{\g}^{\sub}_{\z_t}(f, C)}{\z} \le \ip{\hat{\g}^{\sub}_{\z_t}(f, C)}{\z_t}\}$. 
	 \end{enumerate}
	 
	 \item Finally, return a point $\hat{\z}$ from $F$ that has approximately the minimum value among all solutions in $F$, namely such that $f(\hat{\z}) \le \min_{\z \in F} f(\z) + \e'$. This can be accomplished by asking $|F| - 1$ queries to the $\e'$-approximate value comparison oracle.
	 \end{enumerate}
	\end{algorithm}			
	
	To analyze this algorithm we need the following technical lemma regarding deep points in instances in $\I_{n,d,R,\rho,M}^{deep}$. 
	
	\begin{lemma} \label{lemma:optBall}
		For every instance $(f,C)$ in $\I_{n,d,R,\rho,M}^{deep}$, and for every $0 < \e < 2MR$, there is an $\e$-approximate solution $\z$ such that the ball $B_\infty(\z, \frac{\e \rho}{2MR})$ is contained in $C$.
	\end{lemma}
	
	\begin{proof}
	Let $\z^* = (\x^*,\y^*)$ be an optimal solution for the instance, and let $\bar{\z} = (\x^*, \bar{\y})$ be such that $B_\infty(\bar{\z}, \rho)$ is contained in $C$. For $\alpha = \frac{\e}{2MR}$, we claim that the point $\z := (1-\alpha) \z^* + \alpha \bar{\z}$ has the desired properties. First, by convexity of $C$ we have that the desired ball $B_\infty(\z, \frac{\e \rho}{2MR}) = (1-\alpha) \z^* + \alpha B_\infty(\bar{\z}, \rho)$ is contained in $C$. In addition, since $\z - \z^* = \alpha \cdot (0, \bar{\y} - \y^*)$ and $f$ is $M$-Lipschitz over the integer fibers, we have $$f(\z) \,\le\, f(\z^*) + \alpha M \cdot \|\y^* - \bar{\y}\|_{\infty} \,\le\, f(\z^*) + \e,$$ where the last inequality uses that $\y^*,\bar{\y} \in [-R,R]^d$ and the definition of $\alpha$; so $\z$ is an $\e$-approximate solution. This concludes the proof.  
	\end{proof}

	\begin{proof}[Proof of Theorem \ref{thm:UBOpt}]
%
%
%
%
%
%
%
	 We show that Algorithm \ref{alg:centerpoints} with number of iterations set as  $$T = 2^n (n+d) (d+1) \ln \bigg(\frac{2R}{\min\{\rho',1
  \}}\bigg)\in O\left( 2^n (n+d) d \ln \bigg(\frac{MR}{\min\{\rho,1
  \}\epsilon}\bigg)\right)$$ has the desired properties 
  First, regarding the number of oracle queries performed: in each iteration it performs at most 2 approximate separation/value cut queries, and in Step 2 it performs $|F|-1 \le T$ approximate value comparison queries. In total, the algorithm performs at most $3T$ queries, giving the desired complexity. 
	 
	 Now we show that the algorithm returns an $\e$-optimal solution. For that, it suffices to show that for this value of $T$, the set of feasible solutions $F$ contains an $\frac{\e}{2}$-optimal solution.
 
	 Using Lemma \ref{lemma:optBall}, let $\bar{\z}$ be an $\e'$-approximate solution such that the ball $B_\infty(\bar{\z}, \frac{\e'\rho}{2MR}) = B_\infty(\bar{\z}, 2\rho')$ is contained in $C$. Thus, the ball $B_\infty(\bar{\z}, \rho')$ is contained in $C_{-\rho'}$. Since the cut added to $P_t$, whether in Step (b) or (c), goes through the centerpoint $\z_t$, the mixed-integer volume of $P_t$ is reduced by a factor of at least $(1 - \frac{1}{2^n (d+1)})$ in each iteration (to simplify the notation let $\alpha := \frac{1}{2^n (d+1)}$). The definition of $T$ shows that the last set $P_T$ has mixed-integer volume at most
	  \begin{align}
	  (1-\alpha)^T \mu(P_0) = (1-\alpha)^T (2R)^{n+d} \le e^{-T\alpha}(2R)^{n+d}\leq\left(\min\{\rho',1
  \}\right)^{n+d}\leq \left(\min\{\rho',1
  \}\right)^{d}. \label{eq:volRed}
	  \end{align}
Let $X$ be the intersection of $B_\infty(\bar\z,\rho')$ with the mixed-integer fiber containing $\bar\z$. $X$ has the same structure as an $\ell_\infty$ ball of radius $\rho'$ in $\R^d$, and thus has volume at least $(2\rho')^d$, which is strictly bigger than the right-hand side of \eqref{eq:volRed}. This means that some mixed-integer point from $X$ is cut off by one of the hyperplanes applied by the algorithm. However, such a hyperplane cannot be one added in Step (b), since $B_\infty(\bar{\z}, \rho') \subseteq C_{-\rho'}$ and the cuts in that step are valid for $C_{-\rho'}$. Thus, there is an iteration $t$ that added a Step (c) approximate value cut that cut off a point $\tilde{\z} \in X$. Thus, $\ip{\hat{\g}^{\sub}_{\z_t}(f, C)}{\tilde{\z}} > \ip{\hat{\g}^{\sub}_{\z_t}(f, C)}{\z_t}$. Since this is an $\e'$-approximate value cut, we get that $f(\tilde{\z}) \geq f(\z_t) - \e'$. Since $f$ is $M$-Lipschitz on the fiber containing $\bar\z$ and $\tilde\z$, and the $\ell_\infty$ distance between $\bar{\z}$ and $\tilde{\z}$ is at most $\rho'$, we get
	  \begin{align*}
	  	f(\z_t) \le f(\tilde{\z}) + \e' \le f(\bar{\z}) + \rho' M + \e' \le \OPT + 2 \e' + \rho' M \le \OPT + \frac{\e}{2},
	  \end{align*}
	  where the last inequality uses that $\rho' = \frac{\e' \rho}{4MR} \le \frac{\e'}{M}$ and that $\e' = \frac{\e}{6}$. 
	  
	  This shows that the set of feasible solutions $F$ contains an $\frac{\e}{2}$-approximate solution, namely the above $\z_t$, as desired. This concludes the proof for the mixed-integer setting. 
   
   For the continuous setting with $n=0$, the improved bound follows from using an improved bound on centerpoints due to Grünbaum \cite{Gruenbaum1960}. Specifically, $\alpha$ in the left hand side of (\ref{eq:volRed}) can be taken to be $\frac{1}{e}$, where $e$ is Euler's constant.
	\end{proof}

\subsection{Proof of Lemma \ref{lemma:bitQuery-and-dirDer}}\label{sec: proof of bit approximate}

The proof will be divided into two parts: first, we show how to obtain approximate oracles using bit queries, after which we show how to do the same using inner product sign queries.

\bigskip

\noindent {\bf Obtaining approximate oracles using bit queries.}
{Since $f$ is $M$-Lipschitz with respect to the $\ell_\infty$ norm, any subgradient has $\ell_\infty$ norm at most $M$.}	 Thus, letting $\e' := \frac{\e}{2 (n+d) R}$, we will query the sign and the bits {indexed by the integers} $\lceil\log M\rceil, \lceil\log M\rceil - 1, \ldots, - \lfloor\log \frac{1}{\e'}\rfloor$ of each coordinate of $\g^{\sub}_{\bar{\z}}(f, C)$ (nonnegative integers index the bits before the decimal, and negative integers index the bits after the decimal in the binary representation). This can be done by querying the bits of $\g^{\sub}_{\bar{\z}}(f, C)$ for a total of $(n+d) (\log \frac{M}{\e'} + 2)$ queries -- for each coordinate, one queries $\log M + \log (\frac{1}{\eps})+1$ bits for the desired precision and one additional bit for the sign. This gives a vector $\hat{\g}^{\sub}_{\bar{\z}}(f, C)$
such that $\|\g^{\sub}_{\bar{\z}}(f, C) - \hat{\g}^{\sub}_{\bar{\z}}(f, C)\|_\infty \le \sum_{i > \log \frac{1}{\e'}} \frac{1}{2^i} \le \e'$. 

	Then $\hat{\g}^{\sub}_{\bar{\z}}(f, C)$ is an $\e$-approximate value cut. Let $\g:=\g^{\sub}_{\bar{\z}}(f, C)$ and $\hat{\g} := \hat{\g}^{\sub}_{\bar{\z}}(f, C)$ to simplify notation.  For every $\z \in [-R,R]^{n+d}$ such that $\ip{\hat{\g}}{\z} \geq \ip{\hat{\g}}{\bar{\z}}$ we have by convexity of $f$
	\begin{align}
		f(\z) - f(\bar{\z}) \geq \ip{\g}{\z - \bar{\z}} &= \underbrace{\ip{\hat{\g}}{\z - \bar{\z}}}_{\geq 0} + \ip{\g - \hat{\g}}{\z - \bar{\z}} \notag\\
		&\geq - \|\g - \hat{\g}\|_{\infty} \cdot \|\z - \bar{\z}\|_1 \notag \\
		&\geq - 2 \e' (n+d) R \,=\, - \e, \label{eq:approxCut1}
	\end{align}
	where the second inequality follows from H\"older's inequalty, and so $\hat{\g}$ has the desired property. 
	
	For $\hat{\g}^{\sepp}_{\bar{\z}}(f, C)$, recall that by assumption the separating vector $\g^{\sepp}_{\bar{\z}}(f, C)$ has unit length, and hence $\|\g^{\sepp}_{\bar{\z}}(f, C)\|_{\infty} \le 1$. Then querying the sign and the bits indexed by $0,-1,\ldots, -\log \frac{1}{\e'}$ of each coordinate of $\g^{\sepp}_{\bar{\z}}(f, C)$ we obtain a vector $\hat{\g}^{\sepp}_{\bar{\z}}(f, C)$ such that $\|\g^{\sepp}_{\bar{\z}}(f, C) - \hat{\g}^{\sepp}_{\bar{\z}}(f, C)\|_{\infty} \le \e'$. 
	
	We claim that $\hat{\g}^{\sepp}_{\bar{\z}}(f, C)$ is an $\e$-approximate separation oracle, namely the inequality $\ip{\hat{\g}^{\sepp}_{\bar{\z}}(f, C)}{\z} \le \ip{\hat{\g}^{\sepp}_{\bar{\z}}(f, C)}{\bar{\z}}$ holds for all $\z \in C_{-\e}$. As before, to simplify the notation we use $\g := \g^{\sepp}_{\bar{\z}}(f, C)$ and $\hat{\g} := \hat{\g}^{\sepp}_{\bar{\z}}(f, C)$. For every $\z \in [-R,R]^{n+d}$ we have 
	\begin{align}
		\ip{\hat{\g}}{\z} = \ip{\g}{\z} + \ip{\hat{\g} - \g}{\z} \le \ip{\g}{\z} + \|\hat{\g} - \g\|_{\infty} \cdot \|\z\|_1 \le \ip{\g}{\z} + \e' R (n+d) = \ip{\g}{\z} + \frac{\e}{2}. \label{eq:approxCut11}
	\end{align}
	Now we claim that for every point $\z \in C_{-\e}$ we have $\ip{\g}{\z} \le \ip{\g}{\bar{\z}} - \e$: since the inequality $\ip{\g}{\x} \le \ip{\g}{\bar{\z}}$ is valid for the ball $B(\z, \e) \subseteq C$, we have 
	\begin{align}
	\ip{\g}{\bar{\z}} \geq \max_{\w \in B(0,\e)} \ip{\g}{\z + \w} = \ip{\g}{\z} + \e, \label{eq:approxCut2}
	\end{align}
	proving the claim. Finally, we claim that $\ip{\g}{\bar{\z}} \le \ip{\hat{\g}}{\bar{\z}} + \frac{\e}{2}$: 
	\begin{align}
	\ip{\g}{\bar{\z}} - \ip{\hat{\g}}{\bar{\z}} = \ip{\g - \hat{\g}}{\bar{\z}} \le \|\g - \hat{\g}\|_{\infty} \cdot \|\bar{\z}\|_1 \le \e' R (n+d) = \frac{\e}{2}. \label{eq:approxCut3}
	\end{align}
	Combining inequalities \eqref{eq:approxCut11}-\eqref{eq:approxCut3} proves that the cut $\ip{\hat{\g}}{\z} \le \ip{\hat{\g}}{\bar{\z}}$ is valid for $C_{-\e}$.
	
	To obtain the $\e$-approximate value comparison oracle, since $f$ takes values in $[-U,U]$, it suffices to probe the sign plus $\log \frac{2 U}{\e}$ bits of $f(\bar{\z})$ and $f(\bar{\z}')$ to approximate each of the values within $\pm \frac{\e}{2}$, in which case we can decide whether $f(\bar{\z}) \le f(\bar{\z}') + \e$ or not. This concludes the proof when using bit queries. 


\bigskip

\noindent {\bf Obtaining approximate oracles using inner product sign queries.} This proof largely follows the same steps as for the first part, except for the need for the following result, which may be of independent interest.
\medskip

\begin{lemma}\label{lemma:gradapprox}
	For any $\e \in (0,1)$ and any vector $\g \in \R^d$, using $O(d \log \frac{d}{\e})$ inner product sign queries one can obtain a unit-length vector $\hat{\g} \in \R^d$ such that $\big\|\hat{\g} - \frac{\g}{\|\g\|}\big\| \le \e$. 
\end{lemma}

\begin{proof}
	We prove by induction on the dimension $d$ that for every $\delta \in (0,2)$, with $d \log \frac{8}{\delta}$ inner product sign  queries we can obtain a vector $\hat{\g}$ such that $\big\|\hat{\g} - \frac{\g}{\|\g\|}\big\| \le 2 d \delta$; the lemma then follows by setting $\delta = \frac{\e}{2d}$. 

	Just one query suffices when $d=1$, so consider the base case $d=2$. Perform a binary search as follows: Start with the cone $K_0 = \R^2$, with corresponding angle $2\pi$. In iteration $t$, we maintain a cone $K_t$ containing $\g$ whose angle is half that of $K_{t-1}$ as follows. For each iteration, find a line $\{\x : \ip{\a}{\x} = 0\}$ that cuts $K_t$ into two cones $K_t^L = K_t \cap \{\x: \ip{\a}{\x} \le 0\}$ and $K_t^R = K_t \cap \{\x: \ip{\a}{\x} \geq 0\}$ each with half the angle of $K_t$, i.e. bisecting $K_t$. Ask the query ``Is $\ip{\a}{\g} \geq 0$?'', and if so set $K_{t+1} = K_t^R$, otherwise set to $K_{t+1} = K_t^L$, and repeat the procedure. By construction all the cones $K_t$ contain $\g$, and after $\log \frac{8}{\delta}$ iterations we obtain a cone $K$ with angle $\frac{\delta \pi}{4}$. Let $\hat{\g}$ be any vector in this cone with unit $\ell_2$-norm. For any other $\x \in K$ also of unit norm, we have
	\begin{align*}
	\|\hat{\g} - \x\|_2^2 = 2 - 2\ip{\hat{\g}}{\x} \le 2 - 2 \cos(\delta \pi /4) \le (\delta \pi /4)^2 \le \delta^2,
	\end{align*}
	where the second inequality uses that fact that $\cos(\theta) \geq 1 - \frac{\theta^2}{2}$ for all $\theta \in (0, \pi/2)$. So $\|\hat{\g} - \x\|_2 \le \delta$ for all unit-norm vectors in $K$, and in particular $\hat{\g}$ gives the desired approximation of $\frac{\g}{\|\g\|_2}$, proving the desired result when $d=2$.
	
	\medskip Now consider the general case $d > 2$. Consider any 2-dimensional subspace $A$ of $\R^d$, and let $\Pi_A$ denote the projection onto this subspace. Using the $2$-dimensional case on the subspace $A$, we see that by using $\log \frac{8}{\delta}$ queries of the form ``Is $\ip{\a}{\Pi_A \g} \geq 0$?'', we can obtain a unit length vector $\tilde{\g} \in A$ such that  $\|\lambda_A \cdot \tilde{\g} - \Pi_A \g\| \le \delta \|\Pi_A \g\|$, where $\lambda_A := \|\Pi_A \g\|$.
	 We note that since $\ip{\a}{\Pi_A \g} = \ip{\Pi_A^* \a}{\g}$, the required queries can be obtained by inner product sign  queries (here $\Pi_A^*$ denotes the adjoint linear operator for the projection operator $\Pi_A$, whose matrix representation is given by the transpose of the matrix representing the projection $\Pi_A$).

	Now consider the $(d-1)$-dimensional subspace $B := \mspan\{\tilde{\g}, A^{\perp}\}$, and notice that $\dist(\g, B) \le \delta \|\g\|$: the vector $\b := \lambda_A \cdot \tilde{\g} + (\g - \Pi_A \g)$ belongs to $B$ and $\|\g - \b\| = \|\lambda_A \cdot \tilde{\g} - \Pi_A \g\| \le \delta \|\Pi_A \g\| \le \delta \|\g\|$. Since $\g$ is close to this subspace, we project it there and recurse on dimension. More precisely, consider the projection $\Pi_B \g$ of $\g$ onto $B$, and inductively obtain a vector $\hat{\g} \in B$ such that $\|\lambda_B \cdot \hat{\g} - \Pi_B \g\|_2 \le 2 (d-1) \delta \cdot \|\Pi_B \g\|$ (letting  $\lambda_B := \|\Pi_B \g\|$), by using additional $(d-1) \log \frac{8}{\delta}$ queries (for a total of $d \log \frac{8}{\delta}$ queries). 
	
	We claim that $\hat{\g}$ is the desired approximation of $\g$, namely $\|\frac{\g}{\|\g\|} - \hat{\g}\|_2 \le 2 d \delta$. To see this, from triangle inequality we have
	\begin{align}
		\big\|\g - \|\g\|\cdot  \hat{\g}\big\| \le  \|\g - \Pi_B \g\| + \|\Pi_B \g - \lambda_B \cdot \hat{\g}\| + \|\lambda_B \cdot \hat{\g} - \|\g\| \cdot \hat{\g}\big\|.   \label{eq:dirDerMain}
	\end{align} 
	The first term of the right-hand side equals $\dist(\g, B)$, which is at most $\delta \|\g\|$ as argued above. For the second term, by induction we have $$\|\Pi_B \g - \lambda_B \cdot \hat{\g}\|_2 \le 2 (d-1) \delta \cdot \|\Pi_B \g\| \le 2 (d-1) \delta \|\g\|.$$  Finally, we claim that the last term of \eqref{eq:dirDerMain} is at most $\delta \|\g\|$: since $\hat{\g}$ has unit norm, it equals $|\lambda_B - \|\g\|| = |\|\Pi_B \g\| - \|\g\|| = \|\g\| - \|\Pi_B \g\|$, and by triangle inequality we have $$\|\g\| \le \|\Pi_B \g\| + \|\g - \Pi_B \g\| \le \|\Pi_B \g\| + \delta \|\g\|,$$ giving the claim. 

	 
	  Applying all these bounds to \eqref{eq:dirDerMain}, we get that $\big\|\g - \|\g\|\cdot  \hat{\g}\big\| \le 2 d \delta \|\g\|$, as desired. This concludes the proof of the lemma.
   \end{proof}

Now we are ready to finish the proof. To obtain an $\e$-approximate value comparison oracle using $\mathcal{H}^{\textup{dir}}$, we can do a binary search on the function values using the queries $h^{\val}_{u,c}$ with $u=1$ and different values of $c$ (as the midpoint of the interval in the binary search) -- see Definition~\ref{def:oracle-example}. Thus, with $O(\log \frac{U}{\e})$ queries, we can implement an $\e$-approximate value comparison oracle. 
		
		For the $\e$-approximate separation, we apply Lemma \ref{lemma:gradapprox} above to the separation oracle $\g^{\sepp}_{\bar{\z}}(f, C)$, with  $O\big((n+d) \log \frac{(n+d) R}{\e}\big)$ inner product sign queries to obtain a vector $\hat{\g}^{\sepp}_{\bar{\z}}(f, C)$ such that $\|\g^{\sepp}_{\bar{\z}}(f, C) - \hat{\g}^{\sepp}_{\bar{\z}}(f, C)\| \le \frac{\e}{2R \sqrt{n+d}}$ (recall the non-zero separation oracles are assumed to have unit length). Using the same arguments as in inequalities \eqref{eq:approxCut11}-\eqref{eq:approxCut3}, we see that $\hat{\g}^{\sepp}_{\bar{\z}}(f, C)$ gives a cut valid for $C_{-\e}$, and hence is an $\e$-approximate separation oracle. 

		For the $\e$-approximate value cut oracle, we do the same thing, but apply Lemma \ref{lemma:gradapprox} to  $\g^{\sub}_{\bar{\z}}(f, C)$, with $O\left((n+d) \log \left(\frac{(n+d) MR}{\e}\right)\right)$ oracle calls to obtain an approximation $\|\g^{\sub}_{\bar{\z}}(f, C) - \hat{\g}^{\sub}_{\bar{\z}}(f, C)\| \le \frac{\e}{2M R \sqrt{n+d}}$ and then use the argument from \eqref{eq:approxCut1}.
		
	\medskip
\begin{remark}
    {Notice that} the main ingredient for implementing approximate oracles using inner product sign queries is to use such queries for approximating a given vector $\g$ (Lemma~\ref{lemma:gradapprox}). We remark that this task is related to (actively) learning a linear classifier, namely that whose normal is given by $\g$. While there are existing procedures for doing this, they only guarantee the desired approximation with high probability (albeit on a slightly weaker query model), instead of with probability 1 as we want here; see for example~\cite{activeLearnHalf}. While it is possible that these methods can be adapted to our setting, we present a different, self-contained statement and proof. 
\end{remark}

\section{Proofs of Theorem \ref{thm:binary-UB-finite-transfer} and Corollary~\ref{cor:binary-UB-finite}}\label{sec:binary oracles}

\begin{proof}[Proof of Theorem~\ref{thm:binary-UB-finite-transfer}]
Suppose we have an algorithm $\cA$ that reports an $\eps$-solution to any instance in $\I_{n, d,R,\rho, M}$ after $u$ queries to a full-information first-order oracle based on the first-order chart $\cG$, a finite set of instances $\mathcal{I} \subseteq \I_{n, d,R,\rho, M}$, and a true (unknown) instance $I \in \mathcal{I}$. Our goal is to report a feasible $\eps$-solution using few queries to $\cO(\cG, \cH)$, where $\cH$ contains all binary queries. For this, we design a procedure that maintains a family $\cU \subseteq \I$ of the instances, which always includes the true instance $I$, and possibly determines exact information to pass to $\cA$. We will show that we can always either reduce $|\cU|$ by a constant factor, or determine exact information to use with $\cA$.

Denote by $D$ the query strategy of $\cA$. Initialize $\cU = \mathcal{I}$, and (ordered) lists $Q = \emptyset, H = \emptyset$, which will serve as query-response pairs for the algorithm $\cA$. In particular, $H$ will contain full first-order information about the true instance, and $Q$ will be the sequence of queries $\cA$ makes. While $|\cU| > 1$ and $|Q| \leq u$, do the following:

\begin{itemize}[leftmargin=14pt]
    \item   Set $\z = D(Q, H)$, and query whether $\z$ is feasible. Let us simply write $\g_\z(I)$ to mean $\g^{\sepp}_\z(I)$ if $\z$ is infeasible for the instance $I$, and $\g^{\val}_\z(I)$ or $\g^{\sub}_\z(I)$ if feasible, where $\g^{\sepp}_\z$, $\g^{\val}_\z$ and $\g^{\sub}_\z$ are the first-order maps for separating hyperplane, function value and subgradient, respectively, used by the general binary oracle at $\z$. Write $V_\z$ to be the appropriate codomain. 

    \item \textbf{Case 1:} For every $\vv \in V_\z$
    , at most half of the instances $I'\in \cU$ give $\g_\z(I') = \vv$. Then there exists a set $A \subseteq V_{\z}$ such that the number of instances $I' \in \cU$ with $\g_\z(I')\in A$  is between $\frac{1}{4} |\cU|$ and $\frac{3}{4} |\cU|$. Let $\cU_0 := \{I\in \cU: \g_\z(I')\not\in A\}$, $\cU_1 := \{I\in \cU: \g_\z(I')\in A\}$; thus, $|\cU_i|\leq \frac{3}{4}|\cU|$ for $i=0,1$. 
    Query whether the true instance $I$ has $\g_\z(I)\in A$, using the binary query $h_A: h_A(\vv) = 1$ iff $\vv \in A$, so that $h_A(\g_{\z}(I)) = 0$ if $I \in \cU_0$ and $h_A(\g_{\z}(I)) = 1$ if $I \in \cU_1$.  
    Update $\cU \leftarrow \cU_{q}$, where $q$ is the answer to the query $(\z, h_A)$ given by the oracle. 
    
    \vspace{4pt}
    \item \textbf{Case 2:} There exists $\bar{\vv} \in V_\z$ such that more than half of the instances $I' \in \cU$ have $\g_\z(I') = \bar{\vv}$. Query whether the true instance $I$ has $\g_\z(I) = \bar\vv$, using the binary query $h: h(\bar{\vv}) = 1$ and $h(\x) = 0$ for all other inputs $\x \not = \bar{\vv}$, so that $h(\g_{\z}(I)) = 1$ iff $\g_{\z}(I) = \bar{\vv}$. 
    If $\g_\z(I) \neq \bar\vv$, then update $\cU$ by removing from it all instances $I'$ such that $\g_\z(I') = \bar\vv$, reducing the size of $\cU$ by at least half. Otherwise, if $\z$ was infeasible, we then know the exact separating hyperplane (or function value or subgradient, in the case $\z$ is feasible) for the true instance $I$ and the first-order chart $\cG$, namely $\g^{\sepp}_\z(I) = \bar\vv$, and so employ it to update $Q$ and $H$ by appending $\z$ and $\vv$ to them, respectively, which will serve as information for the algorithm $\cA$.
\end{itemize}

In each step, either the size of $\cU$ decreases by at least $1/4$, or full (exact) first-order information at the query point determined by the query strategy of $\cA$ is obtained and $Q, H$ are updated. The former can only happen $O(\log |\mathcal{I}|)$ times until $\cU$ becomes a singleton, in which case we know the true instance and can report its optimal solution, while if the latter happens $u$ times, one can run the algorithm $\cA$ with the information $(Q,H)$ to report an $\eps$-approximate solution to the true instance $I$, noting that since the points in $Q$ were determined according to the query strategy of the algorithm, the information in $(Q,H)$ is indeed sufficient to run the $\cA$ on for $u$ iterations. Hence, after at most $\log|\cI|+ u$ queries to the general binary oracle, one can report an $\eps$-solution to the true instance.\end{proof}

Corollary \ref{cor:binary-UB-finite} follows immediately when one uses the centerpoint-based algorithm of \cite{oertel2014integer,basu2017centerpoints} as $\cA$, which is the exact oracle version of Algorithm \ref{alg:centerpoints} above and needs at most $$O\left(2^n\,d\, (n+d) \log\left(\frac{dMR}{\min\{\rho,1\}\epsilon}\right)\right)
$$
queries in the mixed-integer case, or
$$O\left(d\log\left(\frac{MR}{\min\{\rho,1\}\epsilon}\right)\right)$$ queries in the continuous ($n=0$) case to produce an $\eps$-approximate solution to any instance in $\cI_{n, d, \rho, M, R}$.


\section{Statements and Declarations}

\paragraph{Competing Interests.} There are no financial or non-financial interests that are directly or indirectly related to this work.

{
\bibliographystyle{plain}
\bibliography{full-bib}
}


\newpage
\appendix 

\section{Information Games}\label{sec: information games}

In this section we define the game-theoretic perspective for information complexity and prove Lemma \ref{lemma:infoAdv}, which we restate for convenience. 

\infoAdv*

Let $\I$ be a family of optimization instances in $\R^n \times \R^d$. Let $\cO$ be an oracle for $\cI$. Let $\epsilon > 0$. We define the {\em information game for $(\I, \cO, \epsilon)$} as a two player game between $\alg$ and $\adv$. The players make moves alternatingly with $\alg$ moving first. For $\alg$, a move is a choice of query $q$ from the oracle $\cO$. For $\adv$, a move is a choice of $r \in H$, where $H$ is the response set of $q$ from the immediately preceding move by $\alg$. The only constraint on $\adv$ is that there should exist at least one instance in $\I$ that will give the same responses as the moves made by $\adv$ in the game so far to queries corresponding to the moves made by $\alg$ so far. More precisely, at any stage of the game, there must exist an instance $I$ such that for every response $r$ given by $\adv$ at any round of the game played so far to an immediately preceding query $q$ made by $\alg$, we have $r=q(I)$.

    Notice that a game strategy for $\alg$ is equivalent to a query strategy from Definition~\ref{def:query-strategy}. Also, any instance $\bar{I} \in \I$ gives a game strategy for $\adv$, by simply reporting the response $q(\bar{I})$ for all queries $q$ made by $\alg$.

    In the game tree defined by the above game, a node is said to be {\em $\epsilon$-unambiguous} if the instances that are consistent with $\adv$'s moves all have a common $\epsilon$-approximate solution; otherwise, the node is {\em $\epsilon$-ambiguous}. The game stops when the game reaches an $\epsilon$-unambiguous node; thus, all leaf nodes of the game tree are $\epsilon$-unambiguous. When the game stops at an $\epsilon$-unambigious node $L$, the {\em loss (or payoff) $\Lag_\epsilon(L)$} is defined as the number of moves made by $\alg$ to arrive at $L$ in the game tree, i.e., it is exactly half of the depth of $L$ in the game tree.

    A subtree $T$ of the game tree is said to be a {\em full subgame tree} if it is precisely the set of all descendants of a node in the game tree (including itself). A subtree $T$ is said to have {\em finite horizon} if it has bounded depth, i.e., there exists $D > 0$ such that all nodes of $T$ have depth (in $T$) bounded by $D$. Otherwise, $T$ is said to have {\em infinite horizon} (note this includes the case where $T$ has an infinite length path and the case where all paths are finite length, but there is no upper bound on the lengths).
    
    We define the value $v(T)$ of a finite horizon subtree $T$ by induction on the depth of $T$, which we call the {\em value of the subgame defined by $T$}.
    
    \begin{itemize}
        \item If $T$ consists of a single node $N$, $v(T):=\Lag_\epsilon(N)$ if $N$ is an $\epsilon$-unambiguous node, else $v(T):=\infty$. 
        \item If $T$ is rooted at a node $N$ corresponding to $\alg$, then $v(T):= \inf\{v(T'): T' \textrm{ child subtree in }T\textrm{ at }N\}$. 
        \item If $T$ is rooted at a node $N$ corresponding to $\adv$, then $v(T):= \sup\{v(T'): T' \textrm{ child subtree in }T\textrm{ at }N\}$. 
    \end{itemize} 
    
    Let $T$ and $T'$ be subtrees of the game tree. $T'$ is said to be a {\em (finite horizon) truncation of $T$} if there exists $D >0$ such that $T'$ consists of all nodes of $T$ with depth bounded by $D$. Observe that if $T, T'$ are finite horizon subtrees and $T'$ is a truncation of $T$, then $v(T') \geq v(T)$. We can thus naturally extend the notion of value to an infinite horizon subtree $T$, as a limit of its finite horizon truncations: $$v(T) := \inf\{v(T'): T' \textrm{ finite horizon truncation of }T\}.$$

    A {\em game strategy} for $\alg$ (resp. $\adv$) is a choice of a specific move at every non leaf node corresponding to $\alg$ (resp. $\adv$). Thus, any game strategy $Q$ for $\alg$ (resp. game strategy $A$ for $\adv$) corresponds to a subtree of the entire game tree, where we select a single outgoing edge at every node corresponding to $\alg$ (resp. $\adv$). Such a subtree is called a {\em decision tree for $\alg$} (resp. $\adv$). With a slight abuse of notation, we will use $v(Q)$ (resp. $v(A)$) to denote the values of these decision trees. A simultaneous choice of strategies $Q, A$ yields a single path in the original game tree; $v(Q,A)$ will denote the value of this path (subtree).

    We note that we can express the notion of an $\e$-hard adversary for $\ell$ using this notation; it follows directly from the definition of the value function $v(\cdot)$.

    \begin{observation} \label{obs:hardAdv}
        An adversary $A$ is $\e$-hard for $\ell$-rounds iff $\inf_Q v(Q,A) > \ell$, where the infimum is taken over all strategies for $\alg$. 
    \end{observation}

    We now focus on proving Lemma \ref{lemma:infoAdv}, which requires a few preparatory observations. The first is the following, which is a useful extension of the recursive nature of the value of a subtree to infinite horizon trees. 

\begin{lemma}\label{lem:value-recursion}
Let $T$ be any subtree of the game tree (not necessarily of finite horizon). Then, \begin{itemize}
    \item If $T$ consists of a single node $N$, $v(T)=\Lag_\epsilon(N)$ if $N$ is an $\epsilon$-unambiguous node, else $v(N)=\infty$. 
    \item If $T$ is rooted at a node $N$ corresponding to $\alg$, then $v(T)= \inf\{v(T'): T' \textrm{ child subtree in }T\textrm{ at }N\}$. 
    \item If $T$ is rooted at a node $N$ corresponding to $\adv$, then $v(T)= \sup\{v(T'): T' \textrm{ child subtree in }T\textrm{ at }N\}$. 
\end{itemize} 
\end{lemma}

\begin{proof}
The result follows from the fact that for any finite horizon truncation $T'$ of $T$, the subtrees of $T'$ rooted at the children of the root of $T'$ (and $T$) are finite horizon truncations of the subtrees of $T$ rooted at the children of the root.
\end{proof}

The following technical property will be used several times in the sequel. 

\begin{lemma}\label{lem:subgame-finite}
The value of the full game is finite if and only if every full subgame has finite value.
\end{lemma}

\begin{proof}
Since the full game is a subgame of itself, if every subgame has finite value so does the full game. For the other direction, suppose now that the full game tree has finite value. Then, there is a finite horizon truncation $T$ with finite value. The next claim shows that for every full subgame tree, its value can be upper bounded using the value of $T$, so in particular is finite as desired.

\begin{claim}\label{claim:subtree-value} Let $S$ be a  full subgame tree, and let $N$ denote its root node. Then $v(S) \le v(T) + depth(N) + 1$.
\end{claim}

\begin{proof} 
Let $D$ be the depth of $T$. Consider the truncation $S'$ of $S$ which has depth $D$ if $N$ is a node corresponding to $\alg$, and depth $D+1$ if $N$ is a node corresponding to $\adv$. Since $S'$ is a truncation of $S$, it follows that $v(S) \leq v(S')$, and so we only need to show $v(S') \leq v(T) + depth(N) + 1$.

    To prove this bound, consider first the case when $N$ is a node corresponding to $\alg$. In this case, $T$ is isomorphic to a subtree of $S'$, i.e., each node of $T$ corresponds to a node in $S'$ in the natural way: the root of $T$ corresponds to the root of $S'$, the children of the root of $T$ corresponds to (some of) the children of the root of $S'$, etc. Observe that the nodes corresponding to $\alg$ have the same possibilities for moves in both trees $S'$ and $T$ (in fact, every node of $\alg$ in the full game tree has the same choice of moves by definition). However, nodes corresponding to $\adv$ have a smaller set of choices in $S'$, i.e., children, compared to $T$, because they are deeper in the game tree. Thus, in the inductive definition of $v(S')$, a supremum computed at a node in $S'$ for $\adv$ is over a smaller set of choices compared to the corresponding supremum when computing $v(T)$. Finally, the depth (in the full game tree) of any node in $S$ is precisely $depth(N)$ more than the depth of the corresponding node in $T$. Thus, $v(S') \leq v(T) + depth(N) \leq v(T) + depth(N) + 1$.

Next, consider the case when $N$ is a node corresponding to $\adv$. All the subtrees of $S'$ rooted at the children of $N$ are again in correspondence with $T$, since $S$ is of depth $D+1$. The previous argument applies to these subtrees and thus, their values are all bounded by $v(T) + depth(N) + 1$. Taking a supremum at node $N$ shows that $v(S') \leq v(T) + depth(N) + 1$.
\end{proof}
\end{proof}

We first prove the existence of a saddle-point in this game (in fact, the same argument gives that there is a subgame perfect equilibrium, which essentially means that this saddle point property holds for every subtree).

\begin{lemma} \label{lemma:equi}
    Suppose the value of the full game is finite (in other words, the optimization problem is solvable with finitely many queries using the given oracles). Then, there exists a saddle-point $Q^\star, A^\star$ for $\alg$ and $\adv$ respectively, i.e., strategies $Q^\star$ and $A^\star$ satisfying
    \begin{align}
    \inf_Q v(Q, A^\star) = v(Q^\star, A^\star) = \sup_A\; v(Q^\star,A) = v(\textup{full game}), \label{eq:equi}
    \end{align}
    where \textup{full game} denotes the full game tree, the infimum on the left-hand side is over all possible strategies for $\alg$, and the supremum on the right-hand side is over all strategies for $\adv$. 
\end{lemma}

\begin{proof}

Since we assume the full game has finite value, by Lemma~\ref{lem:subgame-finite} we know that the value $v(S)$ of any full subgame tree $S$ is also finite. Combined with the fact that $v(S)$ is integer valued, the supremum or infimum in Lemma~\ref{lem:value-recursion} is attained. A saddle-point $Q^\star$ and $A^\star$ can now be defined as follows. At any node $N$ corresponding to a move by $\alg$, select the move that attains the infimum giving the value of the full subgame tree rooted at $N$ by Lemma~\ref{lem:value-recursion}. At any node $N$ corresponding to a move by $\adv$, select the move that attains the supremum giving the value of the full subgame tree rooted at $N$ by Lemma~\ref{lem:value-recursion}. 

We now prove that $Q^\star,A^\star$ satisfy the desired properties. When the full tree is finite, this follows directly from the definition of $Q^\star,A^\star$; the case where it is not finite is that requires more subtle arguments based on finite horizon truncations. Let $T$ denote the full game tree, and for a node $u$ let $T_u$ denote its full subtree rooted at $u$.

We start by proving that $v(Q^\star,A^\star) = v(T)$. Let $p_1,p_2,\ldots$ be the nodes in the path $P^\star$ induced by $Q^\star,A^\star$ ($p_1$ being the root of the full game tree). By the second item of Lemma \ref{lem:value-recursion}, we have that $v(T_{p_1})$ is the smallest value of the subtrees rooted at the children of $p_1$; since the strategy $Q^{\star}$ chooses precisely such child $p_2$ of smallest value, we get $v(T_{p_1}) = v(T_{p_2})$. Similarly, using the same lemma and the definition of the adversarial strategy $A^{\star}$, we get $v(T_{p_2}) = v(T_{p_3})$. Thus, subtrees $T_{p_i}$'s have the same value, equal to $v(T_{p_1}) = v(T)$ (which is finite). Moreover, the value of a subtree $T_u$ is at least half of the depth of $u$: $v(T_u)$ is either $\infty$ or is the (finite) payoff of a node $w$ in $T_u$, which is at least half the depth of $w$, and so at least half the depth of $u$. Therefore, all nodes in the path $P^\star$ are at depth at most $2 \cdot v(T)$; in particular, $P^\star$ has finite length. Now consider a finite horizon truncation $T'$ of $T$ with value $v(T') = v(T)$; we can further increase the depth of this truncation and assume that the path $P^\star$ is contained in $T'$.
Since $T'$ is finite, its value $v(T')$ is defined by taking the child that has subtree of smallest/largest value on the nodes of $\alg$/$\adv$, until a leaf (of finite value) is reached. But this is by construction the choices that the path $P^\star$ makes, namely they reach the same leaf (or leaves of the same value). The values of $v(T')$ and $v(P^\star)$ are then the value of this leaf/leaves, and hence $v(T') = v(P^\star)$, or equivalently $v(Q^\star,A^\star) = v(T)$ as desired. 

    \medskip
We now consider the first desired equation, namely $\inf_Q v(Q, A^\star) = v(Q^\star, A^\star)$, and notice that it suffices to show
    \begin{align}
        \inf_Q v(Q,A^\star) \ge v(Q^\star, A^\star) \label{eq:long}
    \end{align}
(the equality follows by taking $Q = Q^\star$). Assume that the paths induced by $Q,A^\star$ and $Q^\star,A^\star$ are different, else there is nothing to show. Let $u$ be the first node of $\alg$ in these paths where $Q$ and $Q^\star$ make different decisions; let $w$ and $w^\star$ be the children of $u$ on the paths induced by $Q,A^\star$ and $Q^\star,A^\star$ respectively. Notice that by construction of the choice $Q^\star$ makes at $u$, we have $v(T_{w^\star}) \le v(T_w)$. 

    Now we claim that $v(Q,A^\star) \ge v(T_w)$. If the path $P$ induced by $Q,A^\star$ has infinite length, then all of its nodes are $\e$-ambiguous and so by definition $v(Q,A^\star) = \infty$ and the inequality holds. So suppose $P$ has finite length. Since the value $v(T_w)$ is finite, again due to our assumption that $v(T)$ is finite, there is a finite horizon truncation $T'$ of this subtree with the same value; again, we can further increase the depth of this truncation and assume that the path $P \cap T_w$ (the suffix of $P$ starting at $w$) is contained in $T'$ (and hence $P \cap T_w = P \cap T'$). By definition, the nodes of $\adv$ in the path $P \cap T'$ always select the action that lead to the child whose subtree has largest value (being based on $A^\star$), while the nodes of $\alg$ may not necessarily select the children with lowest value subtrees (being based on $Q$). Since in the definition of the value $v(T')$, both the nodes of $\adv$ and $\alg$ make the optimal decisions, we see that $v(T') \le v(P \cap T')$. Since the value of the path $v(P)$ and of its suffix $v(P \cap T')$ are the same (they are the value of the leaf of the path), we obtain $v(T_w) = v(T') \le v(P \cap T') = v(P) = v(Q,A^\star)$, as claimed.

    Next, we claim that $v(T_{w^\star}) = v(Q^\star,A^\star)$. As proved above, the path $P^\star$ induced by $Q^\star,A^\star$ has finite length. Since the value $v(T_{w^\star})$ is finite, as above, consider a truncation $T'$ of $T_{w^\star}$ with same value $v(T') = v(T_{w^\star})$ and that contains the suffix $P^\star \cap T_{w^\star}$. By the same argument as in the previous paragraph, but now using that both $Q^\star$ and $A^\star$ make optimal decisions, we have $v(T') = v(P^\star \cap T_{w^\star})$,  which is also equal to $v(P^\star)$; this implies $v(T_{w^\star}) = v(P^\star) = v(Q^\star,A^\star)$ as desired.

    Putting together the bounds from the three previous paragraphs yields $v(Q^\star,A^\star) = v(T_{w^\star}) \le v(T_{w}) \le v(Q,A^\star)$ for every strategy $Q$ for the algorithm. Taking an infimum over all such $Q$'s finally proves \eqref{eq:long}, and hence that $\inf_Q v(Q, A^\star) = v(Q^\star, A^\star)$.

    \medskip
    The proof that $\sup_A v(Q^\star, A) = v(Q^\star, A^\star)$ uses the exact same arguments, but exchanging the roles of the players $\alg$ and $\adv$.
\end{proof}

By standard arguments, the existence of a saddle-point implies a minimax result, which is the first element required to prove Lemma \ref{lemma:infoAdv}.

\begin{lemma} \label{lemma:minimax}
    Suppose the value of the full game is finite, and let $Q^\star,A^\star$ be a subgame perfect equilibrium guaranteed to exist by Lemma \ref{lemma:equi}. The the following minimax holds; $$\sup_A \inf_Q \;v(Q,A) = v(Q^\star, A^\star) = \inf_Q\; \sup_A\; v(Q,A),$$ where the infima are over all possible strategies for $\alg$, and the suprema are over all strategies for $\adv$. 
\end{lemma}

\begin{proof}
From the guarantees of $Q^\star,A^\star$, for all strategies $Q,A$ for $\alg$ and $\adv$ respectively, we have  $$v(Q^\star, A^\star) \leq \inf_Q v(Q, A^\star) \le \sup_A \inf_Q v(Q, A)$$ and $$\inf_Q \sup_A v(Q,A) \leq \sup_A v(Q^\star, A) \leq v(Q^\star, A^\star).$$ Since we always have $\sup_A \inf_Q v(Q, A) \leq \inf_Q \sup_A v(Q,A)$, we get equalities throughout.  
\end{proof}

 The last element required for proving Lemma \ref{lemma:infoAdv} is the equivalence between the value of this game and the information complexity of the underlying optimization problem. 

\begin{lemma} \label{lemma:equivIcomp}
    The value of the full game $v(\textup{full game})$ equals $\icomp_{\e}(\cI,\cO)$.
\end{lemma}

\begin{proof}
    Combining Lemmas \ref{lemma:equi} and \ref{lemma:minimax}, $v(\textup{full game}) = \inf_Q \sup_A v(Q,A)$, so it suffices to show $$\inf_Q \sup_A v(Q,A) = \inf_Q \sup_{I \in \cI}\icomp\textstyle{_\epsilon}(Q,I, \cO).$$

    To see the ``$\ge$'' direction: Consider any strategy $Q$ for $\alg$ and instance $I \in \cI$. Consider the adversary $A$ based on the instance $I$, namely that reports $q(I)$ whenever $Q$ asks a query $q$. Notice that $v(Q,A) = \icomp_\e(Q,I,\cO)$ (if $\icomp_\e(Q,I,\cO)$ is finite, after exactly $\icomp_\e(Q,I,\cO)$ many moves of the algorithm $Q$ we reach an $\e$-unambiguous node in the path induced by $Q,A$, so $v(Q,A) = \icomp_{\textstyle \e}(Q,I,\cO)$; if  $\icomp_\e(Q,I,\cO)$ is infinite, every finite horizon truncation of the path induced by $Q,A$ must end on an $\e$-ambiguous node, so $v(Q,A)$ is also infinite). That is, for every $I$ we can find a strategy $A$ for the adversary with the ``same value'', which then implies $\sup_A v(Q,A) \ge \sup_{I \in \cI} \icomp_{\e} (Q,I,\cO)$. Taking an infimum over all $Q$'s on both sides gives the desired inequality $\inf_Q \sup_A v(Q,A) \ge \inf_Q \sup_{i \in \cI} \icomp_{\e} (Q,I,\cO)$.

    The proof for the ``$\le$'' direction is analogous: Consider any strategy $Q$ for $\alg$ and strategy $A$ for $\adv$. If $v(Q,A) = \infty$, then for every finite horizon truncation of length $N$ (i.e. with $\frac{N}{2}$ queries by the algorithm) of the path induced by $Q,A$ ends on an $\e$-ambiguous node; let $I \in \cI$ be an instance consistent with $A$ given up to this node. Then we see that $\icomp_\e(Q,I,\cO) \ge \frac{N}{2}$. Since this holds for every even number $N$, we get $\sup_{I \in \cI} \icomp_\e(Q,I,\cO) = \infty$. Similarly, if $v(Q,A)$ is finite, then there is a node $L$ in the path induced by $Q,A$ with $\Lag_\e(L) = v(Q,A)$; any instance $I$ associated to this node $L$ (i.e., compatible with the answers given by $A$ until reaching $L$) then has $\icomp_{\e}(Q,I,\cO) \ge \Lag_\e(L) = v(Q,A)$. These observations imply $\sup_{I \in \cI} \icomp_\e(Q,I,\cO) \ge \sup_A v(Q,A)$. Taking an infimum over $Q$ gives $\inf_Q \sup_{I \in \cI} \icomp_\e(Q,I,\cO) \ge \inf_Q \sup_A v(Q,A)$ as desired. 

    This concludes the proof of the lemma. 
\end{proof}

Now we can finally prove Lemma \ref{lemma:infoAdv}, stated in the beginning of the section.


\begin{proof}[Proof of Lemma \ref{lemma:infoAdv}]
   From Observation \ref{obs:hardAdv}, there is an adversary $A$ that is $\e$-hard for $\ell$-rounds iff $\sup_A \inf_Q v(Q,A) > \ell$. Combining Lemmas \ref{lemma:equi}, \ref{lemma:minimax}, and \ref{lemma:equivIcomp}, the latter is equivalent to  $\icomp_\eps(\I, \cO) > \ell$, which concludes the proof. 
\end{proof}

\end{document}